%% file: main.tex
\documentclass[11pt]{amsart}
\input{preamble.tex}

\title{On the quantum $\sl_3$ invariant of positive links}

\date{\today}
\thanks{
	{\em\noindent2020 AMS Subject Classification:} 57K16 (primary), 57K14, 57K10
	\\
	{\em Key words and phrases:}
	positive knots and links, positive braids, fibered positive knots, state graph, quantum invariant, $\mathfrak{sl}_3$ invariant, web diagrams, web calculus. 
}

\author{Matthew Harper}
\address{Department of Mathematics \\
	Michigan State University \\
	619 Red Cedar Rd.\\
	East Lansing, MI 48824
}
\email{mrhmath@proton.me, harpe111@msu.edu}

\author{Efstratia Kalfagianni}
\address{Department of Mathematics \\
	Michigan State University \\
	619 Red Cedar Rd.\\
	East Lansing, MI 48824
}
\email{kalfagia@msu.edu}

\begin{document}
	
	\begin{abstract} We use the skein theory of $\Atwo$-webs to study the properties of the quantum $\sl_3$-link polynomial of positive links.	We give explicit formulae for the three leading terms of the polynomial on positive links in terms of diagrammatic quantities of their positive diagrams. We show that a positive link is fibered if and only the second coefficient of the polynomial is equal to one. We also show that the third coefficient provides obstructions to representing links by positive braids.
	
	\end{abstract}
	\maketitle
		
	\section{Introduction}
	We study   the Reshetikhin-Turaev \cite{RT} quantum link polynomial, corresponding to the 3-dimensional defining representation of the Lie algebra $\sl_3$, for positive links. The invariant of oriented links is a Laurent polynomial in a single variable $q$. We obtain explicit diagrammatic formulae for the three leading coefficients of this invariant and discuss their applications to knot theory.
Our approach to the invariant is through the theory of  $\mathsf{A}_2$-webs as defined by Kuperberg in \cite{Kuperberg}. 

Throughout the paper, we will usually refer to this invariant as the $\sl_3$ polynomial for links $L \subset S^3$. The particular normalization we work with and  its relation to other normalizations appearing in the literature is discussed in the beginning of the next section.

A link $L$ is \emph{positive} if it admits a diagram containing only positive crossings. It is a classical result that  for a positive link $L$ the surface obtained by applying Seifert's algorithm to
any connected positive diagram $D$ of $L$ has maximum Euler characteristic over all Seifert surfaces of $L$. 
We use $\chi(L)$ to denote this Euler characteristic.  
By construction, this surface contains a special spine, denoted by $\G_W\coloneqq  \G_W(D)$ in this paper, which is called the \emph{Seifert graph} of $D$. 

We will use $\scs$ to denote the number of vertices of 
$\G_W$. Also we will use $\edges'$ to denote the number of edges in the reduced graph $\G'_W$, obtained from $\G_W$
by removing all the multiple edges between all the pairs of vertices.
If $L$ admits a positive non-connected diagram, then all the quantities defined above are additive over the connected components of the diagram.
	
To describe our results, for a positive link  $L \subset S^3$, let
	$$\lla L\rra= \g_1 q^n+ \g_2  q^{n-2}+ \g_3 q^{n-4}+ \textup{lower degree terms}$$
	denote the $\Atwo$ polynomial of $L$ in the variable $q$ with integral coefficients $\g_i$. For positive links we  show $\lla L\rra$ is a polynomial in $\mathbb{Z}[q^2,q^{-2}]$.

The main technical result of this paper is the following, where the terminology used is defined in detail in Section \ref{sec:leadthree}. Specifically, we define what it means for a pair of reduced edges to be mixed at a vertex in Definition \ref{def:mixed}.

	\begin{thm}\label{thm.thirdi} For any connected positive diagram $D=D(L)$, we have:
	\begin{enumerate}[(1)]
	\item  the leading degree of $\lla L\rra$  is  $n=2\chi(L)\leq 2$, and
	if $n=2$ then $L$ is the unknot,
	\smallskip
	\item $\g_1=1$ and $\g_2=\scs-\edges'\leq 1$,
	 \smallskip
	 
	 \item  $\g_3={\frac {(\g_2+1)\,  \g_2}{2}} + \mu -\th  $.
	\end{enumerate}
	Here $\scs$ and  $\edges'$ are as defined above,
		$\mu$ is the number of edges in $\G'_W$ that have multiplicity greater than one  in $\G_W$, and $\th$ is the number of pairs of edges in $\G'_W$  which are {mixed at a vertex} in  $\G'_W$.
\end{thm}

 Theorem \ref{thm.thirdi} and its proof imply that for any non-trivial positive knot, the polynomial $\lla K\rra$ contains only non-positive powers of the variable $q$.
 In the process of proving the theorem we give a state model reformulation of the $\Atwo$ polynomial and study  how the contribution of each state to various terms of the polynomial change under transition between states by skein moves on webs.

Combining our work here with a result of Futer, Kalfagianni and Purcell \cite{Gutsbook} we obtain the following characterization of fibered positive links. 

\begin{thm}\label{thm.fiberi}  A non-split  positive  link $L$ is fibered if and only if
 $\g_2(L)=1$.
\end{thm}

A particularly interesting class of positive links is the class of links that can be represented as closures of positive braids i.e.\!  products of only positive powers  of the Artin generators of braid groups.
It been known for  a long time that closed positive braids are fibered and that not all positive knots are closed positive braids nor are all positive knots fibered. Positive braids have been studied extensively in low-dimensional topology and the question of determining which links can be represented by positive closed braids has been studied considerably in the literature. 
Our results here have applications to this question. For instance we show the following. The reader is referred to Section \ref{sec:braids} for more results and discussion.
	
\begin{cor}\label{cor.braidsi} Let $K$ be a knot that  decomposes into a connected sum of $p$ prime knots.
	If $K$ can be represented as a closed positive braid, then  $\g_2=1$ and $\g_3=p+1$.
	In particular, if a prime knot is a closed positive braid, then $\g_3=2$.
	\end{cor}

The paper is organized as follows.

In Section \ref{sec:defn} we recall the definition of the quantum $\Atwo$  link invariant and specify the conventions and normalization we will adopt throughout the paper.
We also describe a state model approach for computing the invariant and prove a technical lemma for use in subsequent sections. 

In Section \ref{sec:leadtwo} we begin by introducing  \emph{state graphs} which are graphs corresponding to states that we use to  compute the $\Atwo$  link polynomial.
The main result in this section is Theorem  \ref{thm.second} in which we determine  the two leading terms  of the  polynomial for positive links.

In Section \ref{sec:char} we study the question of when a positive link is fibered and we prove Theorem \ref{thm.fiberi}.
In fact we have a stronger version of the theorem, see
Theorem \ref{thm.fiber}.

In Section \ref{sec:leadthree} we prove the main technical result of the paper (Theorem \ref{thm.third}) determining the three leading terms of  the  $\Atwo$  link polynomial  for positive links.
This in particular implies Theorem \ref{thm.thirdi}.

In Section  \ref{sec:split-connect} we study the behavior of the coefficients $ \g_2$ and $ \g_3$ under connected sum and disjoint union of links.

In Section \ref{sec:braids} we discuss obstructions to representing links by positive closed braids and, in particular, we prove Corollary \ref{cor.braidsi}.
We also show that the only non-split links that are represented as positive alternating closed braids  are 
connected sums of  $(2,n)$ torus links.
We also compare our obstructions to previous related work in the literature, most notably the work of Ito \cite{Ito} that has derived obstructions to positive braiding using the 2-variable HOMFLY link polynomial.
\smallskip

\subsection*{Acknowledgement} The research of M.H. is partially supported by  the  NSF/RTG Grant, DMS-2135960, ``Algebraic and Geometric Topology at Michigan State".
The research of E.K. is partially supported by NSF Grant, DMS-2304033.


	\section{The quantum $\sl_3$ invariant of unframed links}\label{sec:defn}

	\subsection{Definition of the invariant}
	
		We consider the Reshetikhin-Turaev quantum invariant of links associated with $(\sl_3,V)$, where $V$ is the defining representation of the Lie algebra of $\sl_3$. Equivalently, $V$ is the representation with highest weight $(1,0)$. We follow the spider or web approach to this invariant introduced by Kuperberg \cite{Kuperberg},
where it is also called the quantum $\mathsf{A}_2$ invariant. This invariant is computed from an oriented link diagram $D$ by first resolving each crossing by the relations
	\begin{align}\label{eq.smoothing}
		\begin{tikzpicture}
			\draw[->] (1,0) to [out=90,in=-90] (0,1);
			\draw[over,->] (0,0) to[out=90,in=-90] (1,1);
		\end{tikzpicture}
		=q^{-2}
		\begin{tikzpicture}
			\draw[->] (0,0) .. controls (.5,.25) and (.5,.75) .. (0,1);
			\draw[->] (1,0) .. controls (.5,.25) and (.5,.75) .. (1,1);
		\end{tikzpicture}-
		q^{-3}
		\begin{tikzpicture}
			\draw[mid>] (0,0) to (.5,.25);
			\draw[mid>] (1,0) to (.5,.25);	
			\draw[mid<] (.5,.25) to (.5,.75);
			\draw[->] (.5,.75) to (0,1);
			\draw[->] (.5,.75) to (1,1);		
		\end{tikzpicture} && 
		\begin{tikzpicture}
			\draw[->] (0,0) to [out=90,in=-90] (1,1);
			\draw[over,->] (1,0) to[out=90,in=-90] (0,1);
		\end{tikzpicture}
		=q^{2}
		\begin{tikzpicture}
			\draw[->] (0,0) .. controls (.5,.25) and (.5,.75) .. (0,1);
			\draw[->] (1,0) .. controls (.5,.25) and (.5,.75) .. (1,1);
		\end{tikzpicture}-
		q^{3}
		\begin{tikzpicture}
			\draw[mid>] (0,0) to (.5,.25);
			\draw[mid>] (1,0) to (.5,.25);	
			\draw[mid<] (.5,.25) to (.5,.75);
			\draw[->] (.5,.75) to (0,1);
			\draw[->] (.5,.75) to (1,1);		
		\end{tikzpicture}\,.
	\end{align}
	
	This produces a $\bZq$-linear combination of webs diagrams, trivalent graphs embedded in $S^2$ where each vertex is either a source or a sink. These diagrams may be simplified according to the square, bubble, and circle moves of Equation \eqref{eq.webrels}. Iterating these moves simplifies any closed web $W$ to a scalar multiple $\lla W\rra$ of the empty diagram and it is independent of the order the moves are applied. Then $\lla D\rra$ is an invariant of the link presented by $D$. The defining relations on webs are as follows:
	\begin{align}
		\label{eq.webrels}
		\begin{tikzpicture}[scale=.75]
			\draw[mid>] (0,0) to (0,1);
			\draw[mid>] (1,1) to (1,0);
			\draw[mid>] (0,0) to (1,0);
			\draw[mid>] (1,1) to (0,1);
			\draw[mid>] (0,0) to (-.5,-.5);
			\draw[mid>] (1.5,-.5) to (1,0);
			\draw[mid>] (-.5,1.5) to (0,1);
			\draw[mid>] (1,1) to (1.5,1.5);
		\end{tikzpicture}
		=
		\begin{tikzpicture}[scale = 0.75,baseline=9]
			\draw[->] (-.5,1.5) .. controls (.5, 1) and (.5,0) ..  (-.5,-.5);
			\draw[,<-] (1.5,1.5) .. controls (.5, 1) and (.5,0) ..  (1.5,-.5);
		\end{tikzpicture}
		+
		\begin{tikzpicture}[scale = 0.75]
			\draw[->] (-.5,1.5) .. controls (0, .5) and (1,.5) ..  (1.5,1.5);
			\draw[,<-] (-.5,-.5) .. controls (0,.5) and (1,.5) ..  (1.5,-.5);
		\end{tikzpicture}\,,
		\quad\quad
		\begin{tikzpicture}[scale = 0.75]
			\draw[->] (0,0) to (0,.5);
			\draw[] (0,.5) to (0,1);
			\draw[] (0,1) .. controls (-.5,1.5) .. (0,2);
			\draw[] (0,1) .. controls (.5,1.5) .. (0,2);
			\draw[->] (0,2) to (0,2.5);
			\draw[] (0,2.5) to (0,3);		
		\end{tikzpicture}
		=[2]~
		\begin{tikzpicture}[scale = 0.75]
			\draw[->] (0,0) to (0,1.5);
			\draw[] (0,1.5) to (0,3);
		\end{tikzpicture}\,,
		\quad\quad
		\begin{tikzpicture}[scale = 0.5]
			\draw[->] (0,0) arc(0:360:1);
		\end{tikzpicture}
		~&=~
		\begin{tikzpicture}[scale = 0.5]
			\draw[<-] (0,0) arc(0:360:1);
		\end{tikzpicture}
		~=[3],
	\end{align}
	where $[2]=q+q^{-1}$ and $[3]=q^2+1+q^{-2}$.
	
	Our conventions differ from both Kuperberg and Ohtsuki \cite{Kuperberg,Ohtsuki}. First, we use the unframed normalization of the invariant, meaning that it respects the first Reidemeister move. The framing factors in their respective normalizations are $q^{-4/3}$ and $A^8$. After adjusting, we recover our convention from Kuperberg's by replacing $q$ with $-q^{-2}$, and setting $A=-q^{1/3}$ for Ohtsuki's convention.
	
	 If $P_L(a,z)$ denotes the HOMFLY polynomial, defined by the relations 
	\begin{equation}\label{eq.skein}
		a^{-1}P_{L_+}(a,z)-aP_{L_-}(a,z)=zP_{L_0}(a,z) \qquad \mbox{and} \qquad P_{\mathsf{unknot}}(a,z)=1
	\end{equation}
	then
	\begin{equation}\label{eq.homfly}
		\lla L\rra = [3]\cdot P_L(q^{-3},q-q^{-1})\,.
	\end{equation}
	\medskip
	\subsection{Computation of $\lla\cdot\rra$ by states}~
	In this subsection we discuss a state sum approach  for computing  the  quantum $\Atwo$ invariant.

	\begin{defn}An oriented link $L$ is called \emph{positive} (resp.\! \emph{negative}) if it has a diagram $D=D(L)$ in which all crossings are positive (resp.\! negative),
	as shown in Figure \ref{f.pnc}. Note that if a link is positive then its mirror image is negative and vice-versa.
	\end{defn}
	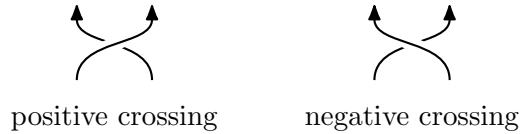
\begin{figure}[h!]
		\begin{tikzpicture}[scale=1]
			\draw[->] (1,0) to [out=90,in=-90] (0,1);
			\draw[over,->] (0,0) to[out=90,in=-90] (1,1);
			\node at (.5,-.5) {positive crossing};
		\end{tikzpicture}
		\quad\quad 
		\begin{tikzpicture}[scale=1]
			\draw[->] (0,0) to [out=90,in=-90] (1,1);
			\draw[over,->] (1,0) to[out=90,in=-90] (0,1);
			\node at (.5,-.5) {negative crossing};
		\end{tikzpicture}
		\caption{A positive crossing and a negative crossing.}\label{f.pnc}
	\end{figure}  
	
	At each crossing of $D$ we may choose either an oriented resolution or a web resolution as shown in Figure \ref{fig.res}. We abbreviate these as $O$ and $W$ resolutions.
		\begin{figure}[h!]
		\begin{tikzpicture}
			\draw[->] (0,0) .. controls (.5,.25) and (.5,.75) .. (0,1);
			\draw[->] (1,0) .. controls (.5,.25) and (.5,.75) .. (1,1);
			\node at (.5,-.5) {$O$-resolution};
		\end{tikzpicture}
		\quad\quad 
		\begin{tikzpicture}
			\draw[mid>] (0,0) to (.5,.25);
			\draw[mid>] (1,0) to (.5,.25);	
			\draw[mid<] (.5,.25) to (.5,.75);
			\draw[->] (.5,.75) to (0,1);
			\draw[->] (.5,.75) to (1,1);		
			\node at (.5,-.5) {$W$-resolution};
		\end{tikzpicture}
		\caption{The oriented ($O$) and web ($W$)  -resolutions of a crossing.}
		\label{fig.res}
	\end{figure}
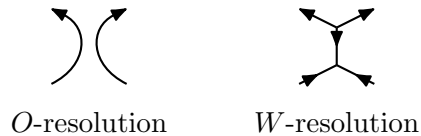
	 A \emph{state} $s\coloneqq s(D)$ is a choice of resolution at each crossing in $D$. The number of states associated to $D$ is $2^{e(D)}$, where $e(D)$ is the number of crossings in $D$.

	 	  There is a bijection between states of $D$ and planar webs obtained by replacing all crossings of $D$ with $O$ and $W$ resolutions, each linear term in the expression of $D$ after applying Equation \eqref{eq.smoothing}. The result of applying state $s$ to $D$ gives a web $\W(s)\coloneqq \W(s)(D)$ that we will call \emph{the web associated to} $s$. There are states for which $\W(s)$ may have several components, including disjoint circles. We denote the \emph{all-$O$ state} by $O\coloneqq O(D)$ and the \emph{all-$W$ state} by $W\coloneqq W(D)$.  We will use $\scs\coloneqq \scs(D)$ to denote the number of circles in $O$ and $\xings\coloneqq \xings(D)$ to denote the number of crossings in $D$. The evaluation of $\W(s)$ using the relations of Equation \eqref{eq.webrels} to an element of $\bZ[q,q^{-1}]$ is denoted $\lla \W(s)\rra$.  
	 
Given a state $s$,	on a fixed diagram $D$,  let $\a_+(s)$ and $\b_+(s)$ (resp.\! $\a_-(s)$ and $\b_-(s)$) denote the number of $O$ and $W$ resolutions of positive (resp.\! negative) crossings in a state $s$ of $D$, respectively. If $D$ is a positive or negative diagram then we simply write $\a(s)$ and $\b(s)$.
	 
	 For each state $s$ we call the quantity
	  $$\phi(s)\coloneqq (-1)^{\b_+(s)+\b_-(s)}q^{-2(\a_+(s)-\a_-(s))-3(\b_+(s)-\b_-(s))},$$ the \emph{phase of $s$}. The phase is computed as the product of coefficients appearing in Equation \eqref{eq.smoothing} over all crossings in a given state. For positive and negative diagrams, 
	  $$\phi(s)=(-1)^{\b(s)}q^{-2\a(s)-3\b(s)} \ \ \ {\rm and} \ \ \phi(s)=(-1)^{\b(s)}q^{2\a(s)+3\b(s)},$$ respectively.
	  The discussion above implies the following.
	 \begin{prop}
	 	Let $D$ be a diagram of a link $L$. Then the quantum $\Atwo$ invariant of $L$ is computed as a sum of state contributions
	 	\begin{align*}
	 		\lla D\rra = \sum_s Y(s), \ \ \mbox{where \ \ \  $Y(s)\coloneqq \phi(s)\lla \W(s)\rra$.}
	 	\end{align*}
	 \end{prop}

	To facilitate our exposition we will refer to the Laurent polynomial $Y(s)$ as the \emph{weight} of the state $s$.
	
	We close the subsection with the following lemma.
	
	\begin{lem}[{\cite[Lemma B.1]{Ohtsuki}}] \label{bigonsquare}
	Let $\W$ be a web embedded on a 2-sphere $S^2$. 	 Then, at least one of the regions of $S^2\setminus \W$ is a bigon or a square.
	\end{lem}	
		
	

\subsection{Relating maximum degrees of weights} Our goal in this subsection is to understand how the maximum degree of a weight $Y(s)$ changes when we replace a single $O$-resolution in  $s$ by the $W$-resolution. Proposition \ref{prop.dec} below is important for the proof of the main technical results of the paper
(Theorems \ref{thm.second} and \ref{thm.third})  that provide  formulae for the three leading coefficients
for  the $\Atwo$ invariant of positive links.
We note however that Lemma \ref{lem.webdec}, that is used for the proof of Proposition \ref{prop.dec}, applies to all link diagrams and not just positive ones.

\begin{defn} We say that two states $s$, $s'$ of a link diagram $D$ are related by an \emph{$OW$-move}
if $s'$ is obtained from $s$ by replacing a single $O$-resolution by a $W$-resolution. If the states  $s$ and $s'$ are related by an $OW$-move, the corresponding webs $\W(s)$ and $\W(s')$ differ only in
a disc $E$ that intersects  $\W(s)$ in two coherently oriented arcs. Then $E\cap \W(s')$ is the replacement of the apparent $O$-resolution with a $W$-resolution.
\end{defn}
A simple example of an $OW$-move is given in Figure \ref{fig.OW}.
	
	\begin{figure}[h!]
		\[
		\begin{tikzpicture}[scale=0.5]
			\draw[mid>] (0,0) arc(0:360:1);
			\draw[->] (3,0) arc(360:0:1);
			\draw[dashed] (1.5,0) arc(0:360:1);
			\node[above] at (.5,1) {$E$};
			\node[below] at (.5,-1) {\phantom{$E$}};
		\end{tikzpicture}
		\xmapsto{\mbox{$OW$-move}}
		\begin{tikzpicture}
			\draw[dashed] (1,.5) arc(0:360:.5);
			\node[above] at (.5,1) {$E$};
			\node[below] at (.5,0) {\phantom{$E$}};
			\draw[] (0,0) to (.5,.25);
			\draw[] (1,0) to (.5,.25);	
			\draw[mid<] (.5,.25) to (.5,.75);
			\draw[] (.5,.75) to (0,1);
			\draw[] (.5,.75) to (1,1);		
			\draw[mid<] (0,0) arc(270:90:.5);
			\draw[mid<] (1,0) arc(-90:90:.5);
		\end{tikzpicture}
		\]\vspace{-1\baselineskip}
		\caption{An $OW$-move between two state webs  $\W$(left) and $\W'$(right)  in the disc $E$ whose boundary is indicated by the dotted line.}\label{fig.OW}
			\end{figure}
	
	Let $\deg(\lla \W(s)\rra )$ denote the maximum degrees of the $\Atwo$ invariant of the web associated to $s$.
	
	\begin{lem}\label{lem.webdec}
		Suppose $s$ and $s'$ are two states of $D$ such that $s'$ is obtained from $s$ by an $OW$-move.
		 Then either
		\begin{align*}
			\deg(\lla \W(s')\rra )=\deg(\lla \W(s)\rra)+1 
			&& \mbox{or} &&  
			\deg(\lla \W(s')\rra )=\deg(\lla \W(s)\rra )-1\,.
		\end{align*}
	\end{lem}

		We simplify notation by writing these two equalities as $\deg( \W(s') )=\deg( \W(s) )+\{-1,1\}$.
		
	\begin{proof} The planar webs	$\W(s)$ and $\W(s')$ only differ in a disc  neighborhood $E$ of the crossing of $D$ where the $OW$-move is applied.
We will refer to the regions of $S^2 \setminus  \W(s)$ as polygons in $\W(s)$.
By Lemma \ref{bigonsquare}, at least one polygon in $ \W(s)$ is either a bigon or a square.	Simplify each web $\W(s)$ and $\W(s')$ outside of $E$ using Equation \eqref{eq.webrels} until no bubbles or squares remain except those that have an edge
that is a component of $E\cap \W(s)$.
Note that the contribution of each of these simplifications to $\deg(\lla \W(s)\rra )$ is the same as the contribution to $\deg(\lla \W(s')\rra )$.
After these simplifications we have pairs of webs that
\begin{enumerate} [(i)]
\item  are related by an $OW$-move inside the disc $E$, and
\smallskip
\item  any bubble or square in the webs must have an edge in $E$.
\end{enumerate}
Without loss of generality assume that the webs  $\W(s)$ and $\W(s')$ are a pair that satisfy (i) and (ii).

We prove the lemma by induction on the number $n$ of edges in $\W(s)$.   We refer to the center region of $E$ in $\W(s)$, between the two oriented arcs as the \emph{clearing region}. See Figure \ref{fig.clearing}. 
	 
	 	\begin{figure}[h!]
	 	\begin{tikzpicture}
	 		\draw[->] (0,0) arc(-45:45:.7);
	 		\draw[->] (1,0) arc(225:135:.7);
	 		\draw[dashed] (0,0) arc(225:585:.7);
	 		\fill[nearly transparent]
	 		(0,0) arc(225:315:.7) --
	 		(1,0) arc(225:135:.7) --
	 		(1,1) arc(45:135:.7) --
	 		(0,1) arc(45:-45:.7);
	 		\node[above] at (.5,1.2) {$E$};
	 	\end{tikzpicture}
	 	\caption{The clearing region between the two arcs of an $O$-resolution. }
	 	\label{fig.clearing}
	\end{figure}

The base case for the induction is $n=2$, seen in Figure \ref{fig.OW}.\ 
In this case,  we have $\lla \W(s)\rra=[3]^2={(q^2+1+q^{-2})}^2$ and $\lla \W(s')\rra=[2][3]=(q+q^{-1})(q^2+1+q^{-2})$, and hence $\deg( \W(s') )=\deg( \W(s) )+\{-1,1\}$\,.	
		
Suppose, inductively, that  the lemma holds for all pairs  $\W(s)$ and $\W(s')$ that satisfy (i) above and
$\W(s)$ has fewer than $n$ edges.
		
For the inductive step, suppose now, that   $\W(s)$ has exactly $n$ edges.  
The proof is  divided into cases according to whether the two arcs of 
$E\cap \W(s)$ belong to different components, say $\W_1$, $\W_2$, of $\W(s)$ (Case (A)) or
to the same component of $\W(s)$ (Case (B)).  Case (A) is easier  to handle since
each of $\W_1, \W_2$ have less that $n$ edges and induction  applies. Case (B) requires a more elaborate analysis:  By our earlier assumptions,
the only bigons or squares in $\W(s)$ must have some  edges on $\W(s)\cap E$.  The case where a bigon or square has two edges in
$\W(s)\cap E$ is considered separately ((B1) and (B2)).  Subcase(B3) deals with the  case where the only bubbles in $\W(s)$ have exactly  one edge
on $\W(s)\cap E$, and it further breaks  down to possibilities (B3a)  and (B3b) according to how the regions bordering the bubble in $\W(s)$ are transformed in $\W(s')$.
Similarly,  Subcase(B4) deals with the  case where the only squares in $\W(s)$ have exactly  one edge
on $\W(s)\cap E$, and it further breaks  down to possibilities (B4a)  and (B4b) according to how the regions bordering the square in $\W(s)$ are transformed in $\W(s')$.\newline

	\noindent \hypertarget{lem.webdec.A}{\textbf{Case(A).}} Suppose that the two arcs of $\W(s)\cap E$ belong on different components,
say $\W_1$, $\W_2$ of $\W(s)$. Each of $\W_1$, $\W_2$ can be simplified entirely outside $E$, using the relations of Equation \eqref{eq.webrels}
until we arrive  at two circles each containing one of the two arcs of $E\cap \W(s)$.
Hence, $\lla \W(s)\rra$ simplifies as a product $c_1c_2[3]^2$ for some $c_1,c_2\in \bZ[q,q^{-1}]$. Similarly $\W(s')$ can be simplified completely away from a resulting bubble intersecting $E$, which determines $\lla \W(s')\rra=c_1c_2[2][3]$. We conclude that $$\deg(\W(s'))=\deg(\W(s))+\{-1,1\}.$$
		\noindent \textbf{Case(B).} Suppose that the arcs of $\W(s)\cap E$ are on the same component of $\W(s)$ and consider the polygon, say $P$,
that contains the clearing region. Note that $P$  must have an even number of vertices to ensure biparticity of the web  (recall that each vertex must be a source or a sink). There are now subcases according to whether the polygon $P$ is a bubble,  a square, or an $m$-gon for $m\geq 6$.\\
				
		\noindent \textbf{Subcase(B1).} If the   polygon $P$  is a bubble, then the local diagrams at $E$ are related by 
		\[
		\W(s')=\begin{tikzpicture}[scale = 0.5]
			\draw[mid>] (0,-1) to (0,0);
			\draw[mid<] (0,0)  .. controls (-.5,1) .. (0,2);
			\draw[mid<] (0,0) .. controls (.5,1) .. (0,2);
			\draw[mid>] (0,2) to (0,3);
			\draw[mid<] (0,3)  .. controls (-.5,4) .. (0,5);
			\draw[mid<] (0,3) .. controls (.5,4) .. (0,5);		
			\draw[->] (0,5) to (0,6);
		\end{tikzpicture}
		=[2]~
		\begin{tikzpicture}[scale = 0.5]
			\draw[mid>] (0,-1) to (0,0);
			\draw[mid<] (0,0)  .. controls (-.75,2.5) .. (0,5);
			\draw[mid<] (0,0)  .. controls (.75,2.5) .. (0,5);
			\draw[->] (0,5) to (0,6);
		\end{tikzpicture}
		=
		[2]\W(s)
		\]
		Thus, $\deg(\W(s'))=\deg(\W(s))+\{-1,1\}$\,.\newline 
		
		\noindent \textbf{Subcase(B2).}
		If  the   polygon $P$ is a square, then, up to a vertical reflection, the local diagrams at $E$ are given by 
		\[
		\W(s')=\begin{tikzpicture}[scale = 0.5]
			\draw[mid>] (0,-1) to (0,0);
			\draw[mid<] (0,0)  to (-1,1);
			\draw[mid>] (-1,1) to (0,2);
			\draw[mid<] (0,0) to (1,1);
			\draw[mid>] (1,1) to (0,2);
			\draw[mid<] (0,2) to (0,3);
			\draw[mid>] (0,3) .. controls (.5,4) .. (0,5);
			\draw[mid>] (0,3) .. controls (-.5,4) .. (0,5);		
			\draw[mid<] (0,5) to (0,6);
			\draw[->] (-1,1) to (-1.5,1);
			\draw[->] (1,1) to (1.5,1);
		\end{tikzpicture}
		=[2]~
		\begin{tikzpicture}[scale = 0.5]
			\draw[mid>] (0,-1) to (0,0);
			\draw[mid<] (0,0)  to (-1,1);
			\draw[mid>] (-1,1) to (0,2);
			\draw[mid<] (0,0) to (1,1);
			\draw[mid>] (1,1) to (0,2);
			\draw[mid<] (0,2) to (0,3);
			\draw[->] (-1,1) to (-1.5,1);
			\draw[->] (1,1) to (1.5,1);
		\end{tikzpicture}
		=
		[2]\W(s)
		\]
		We also have $\deg(\W(s'))=\deg(\W(s))+\{-1,1\}$ in this case. \newline 	
		
		\noindent \textbf{Subcase(B3).} Suppose $P$ is an $m$-gon with $m\geq 6$ and is adjacent to a bubble with one of its edges on one of the components
		of $E\cap \W(s)$. 	
		Applying the $OW$-move, splits $P$ into four polygons in $\W(s')$. One of these regions is a square, with one edge that comes from the edge of  the bubble of $\W(s)$ outside
		$E$.  We will split the argument into two  cases, according to whether one of the remaining three polygons is also a square or not.
		In the diagrams below, the dotted edge indicates additional vertices and outgoing edges of the polygon. \newline 
		
		\noindent \textbf{Subsubcase(B3a).}
		Suppose that  the $OW$-move produces at least one square in $\W(s')$
		that  has no edges that come  from the bubble of $\W(s)$.
		We denote by  $\overline{\W}$ the web obtained by resolving the bubble in $\W(s)$ and omitting the factor $[2]$.  
	We have
		\[
		\W(s)=\begin{tikzpicture}[scale=.5]
			\draw[mid>] (0,-1) to (0,0);
			\draw (0,0) to (.25,.25);
			\draw[dotted] (.25,.25) to (.75,.75);
			\draw (.75,.75) to (1,1);
			\draw[mid>] (1,1) to (1,3);
			\draw[mid<] (1,3) to (0,4);
			\draw[mid>] (0,4) to (-1,3);
			\draw[mid<] (0,0) to (-1,1);
			\draw[mid>] (-1,1) .. controls (-1.5,2) .. (-1,3);
			\draw[mid>] (-1,1) .. controls (-.5,2) .. (-1,3);
			\draw[->] (1,1) to (1.5,.5);
			\draw[mid<] (1,3) to (1.5,3.5);
			\draw[->] (0,4) to (0,5);
		\end{tikzpicture}
		=
		[2]
		\begin{tikzpicture}[scale=.5]
			\draw[mid>] (0,-1) to (0,0);
			\draw (0,0) to (.25,.25);
			\draw[dotted] (.25,.25) to (.75,.75);
			\draw (.75,.75) to (1,1);
			\draw[mid>] (1,1) to (1,3);
			\draw[mid<] (1,3) to (0,4);
			\draw[->] (0,4) to (-1,3);
			\draw[-<] (0,0) to (-1,1);
			\draw (-1,1) to[out=135,in=225] (-1,3);
			\draw[->] (1,1) to (1.5,.5);
			\draw[mid<] (1,3) to (1.5,3.5);
			\draw[->] (0,4) to (0,5);
		\end{tikzpicture}
		=
		[2]\overline{\W}
		\]
		and $\deg(\W(s))=\deg(\overline{\W})+1$. 
		Note that  that  $\overline{\W}$ may still contain
		a square or a bubble with one edge  the second arc of  $E\cap \W(s)$, but this will not affect our analysis below. 
		
		On the other hand,  applying the $OW$-move  in the neighborhood of $E$ in $\W(s)$ we have
		
		\[
		\W(s')=\begin{tikzpicture}[scale=.6]
			\draw[mid>] (0,-1) to (0,0);
			\draw (0,0) to (.25,.25);
			\draw[dotted] (.25,.25) to (.75,.75);
			\draw (.75,.75) to (1,1);
			\draw[mid<] (1,3) to (0,4);
			\draw[mid>] (0,4) to (-1,3);
			\draw[mid<] (0,0) to (-1,1);
			\draw[mid>] (-1,1) .. controls (-1.5,2) .. (-1,3);
			\draw[->] (1,1) to (1.5,.5);
			\draw[mid<] (1,3) to (1.5,3.5);
			\draw[->] (0,4) to (0,5);
			\draw[mid>] (-1,1) to (0,1.5);
			\draw[mid>] (1,1) to (0,1.5);
			\draw[mid<] (0,1.5) to (0,2.5);
			\draw[mid>] (0,2.5) to (-1,3);
			\draw[mid>] (0,2.5) to (1,3);
		\end{tikzpicture}
		=
		\begin{tikzpicture}[scale=.6]
			\draw[mid>] (0,-1) to (0,0);
			\draw (0,0) to (.25,.25);
			\draw[dotted] (.25,.25) to (.75,.75);
			\draw (.75,.75) to (1,1);
			\draw[mid>] (1,1) to (1,3);
			\draw[mid<] (1,3) to (0,4);
			\draw[->] (0,4) to (-1,3);
			\draw[-<] (0,0) to (-1,1);
			\draw (-1,1) to[out=135,in=225] (-1,3);
			\draw[->] (1,1) to (1.5,.5);
			\draw[mid<] (1,3) to (1.5,3.5);
			\draw[->] (0,4) to (0,5);
		\end{tikzpicture}
		~
		+
		~
		\begin{tikzpicture}[scale=.6]
			\draw[mid>] (0,-1) to (0,0);
			\draw (0,0) to (.25,.25);
			\draw[dotted] (.25,.25) to (.75,.75);
			\draw (.75,.75) to (1,1);
			\draw[mid<] (1,3) to (0,4);
			\draw[->] (1,1) to (1.5,.5);
			\draw[mid<] (1,3) to (1.5,3.5);
			\draw[->] (0,4) to (0,5);
			\draw[mid<] (0,0) arc(180:90:1);
			\draw[mid<] (0,4) arc(180:270:1);
		\end{tikzpicture}
		=
		\overline{\W} +[2]
		\underbrace{\begin{tikzpicture}[scale=.6]
				\draw[mid>] (0,-1) to (0,0);
				\draw (0,0) to (.25,.25);
				\draw[dotted] (.25,.25) to (.75,.75);
				\draw (.75,.75) to (1,1);
				\draw[->] (1,1) to (1.5,.5);
				\draw[mid<] (0,0) arc(180:90:1);
				\draw[mid>] (1.5,3.5) .. controls (1,3) and (0,3.5) .. (0,4);
		\end{tikzpicture}}_{\overline \W'}
		=\overline{\W}+[2]\overline{\W}'\,.
		\]
As discussed, $P$ is split into four regions in  $\W(s')$. The middle left of these regions, say $S_1$,  is a square that comes from a bubble of $\W(s)$.	By assumption, at least one of the  remaining  two regions is a square. This is the top square, say $S_2$,  with four edges  indicated in solid lines in  $\W(s')$.
	In the first equality, we apply the  square relation of  Equation \eqref{eq.webrels} to resolve $S_1$ in   $\W(s')$ and transform $S_2$ into a bubble.
	We next apply the  bubble relation of  Equation \eqref{eq.webrels} to further simplify the web.
	Thus,  we get $\deg(\W(s'))=\max(\deg(\overline{\W}), \deg(\overline{\W}')+1)$. 
	
	Now notice that by performing an $OW$-move on  $\overline{\W}'$ in a neighborhood disc if the dashed arc illustrated in the left hand side panel of the equation below, we obtain $\overline{\W}$:
		\[
	\overline{\W}'=
	\begin{tikzpicture}[scale=.5]
		\draw[mid>] (0,-1) to (0,0);
		\draw (0,0) to (.25,.25);
		\draw[dotted] (.25,.25) to (.75,.75);
		\draw (.75,.75) to (1,1);
		\draw[->] (1,1) to (1.5,.5);
		\draw[mid<] (0,0) arc(180:90:1) node[midway] (B) {};
		\draw[mid>] (1.25,3.5) .. controls (.75,3) and (-.25,3.5) .. (-.25,4) node[midway] (A) {};
		\draw[dashed,bend right=20] (A.south) to (B.north);
	\end{tikzpicture}
	\xmapsto{\mbox{$OW$-move}}
	\begin{tikzpicture}[scale=.5]
		\draw[mid>] (0,-1) to (0,0);
		\draw (0,0) to (.25,.25);
		\draw[dotted] (.25,.25) to (.75,.75);
		\draw (.75,.75) to (1,1);
		\draw[mid>] (1,1) to (1,3);
		\draw[mid<] (1,3) to (0,4);
		\draw[->] (0,4) to (-1,3);
		\draw[-<] (0,0) to (-1,1);
		\draw (-1,1) to[out=135,in=225] (-1,3);
		\draw[->] (1,1) to (1.5,.5);
		\draw[mid<] (1,3) to (1.5,3.5);
		\draw[->] (0,4) to (0,5);
	\end{tikzpicture}
	=\overline{\W}\,.
	\]
	Since $\overline{\W}'$ has fewer than $n$ edges we may apply the inductive hypothesis and so $\deg(\overline{\W}')=\deg(\overline{\W})+\{-1,1\}$. Because the coefficients in $q$ for any crossingless web are nonnegative integers, we now have 
		\begin{align*}
			\deg(\W(s'))&=\max(\deg(\overline{\W}),\deg(\overline{\W})+\{0,2\})=\deg(\overline{\W})+\{0,2\}\\
			&=(\deg(\W(s))-1)+\{0,2\}=\deg(\W(s))+\{-1,1\}\,.
		\end{align*}
		
		\noindent \textbf{Subsubcase(B3b).}
		 Suppose that only one of the four regions produced from $P$ in
		   $\W(s')$ is a square  (that is only the one that comes from a bubble of $\W(s)$).
		 In this case, the polygon $P$ must have  at least  eight sides as illustrated in the figure below.
		 We again declare $\overline{\W}$ to be the simplified web without the factor $[2]$. 
		\[
		\W(s)=
		\begin{tikzpicture}[scale=.4]
			\draw[mid>] (0,-1) to (0,0);
			\draw (0,0) to (.25,.25);
			\draw[dotted] (.25,.25) to (.75,.75);
			\draw (.75,.75) to (1,1);
			\draw[mid>] (1,1) to (1,3);
			\draw[mid<] (0,0) to (-1,1);
			\draw[mid>] (-1,1) .. controls (-1.5,2) .. (-1,3);
			\draw[mid>] (-1,1) .. controls (-.5,2) .. (-1,3);
			\draw[->] (1,1) to (1.5,.5);
			\draw[mid<] (1,3) to (1.5,3.5);
			\draw[mid<] (0,4) to (0,5);
			\draw[] (-1,3) to (-.75, 3.75);
			\draw[] (-.75, 3.75) to (0,4);
			\draw[] (.75, 3.75) to (0,4);
			\draw[] (.75,3.75) to (1,3);
			\draw[->] (-.75, 3.75) to (-1.25, 4.5);
			\draw[->] (.75, 3.75) to (1.25, 4.5);
		\end{tikzpicture}
		=
		[2]~
		\begin{tikzpicture}[scale=.4]
			\draw[mid>] (0,-1) to (0,0);
			\draw (0,0) to (.25,.25);
			\draw[dotted] (.25,.25) to (.75,.75);
			\draw (.75,.75) to (1,1);
			\draw[mid>] (1,1) to (1,3);
			\draw (0,0) to (-1,1);
			\draw[mid<] (-1,1) ..controls (-1.25,2).. (-1,3);
			\draw[->] (1,1) to (1.5,.5);
			\draw[mid<] (1,3) to (1.5,3.5);
			\draw[mid<] (0,4) to (0,5);
			\draw[] (-1,3) to (-.75, 3.75);
			\draw[] (-.75, 3.75) to (0,4);
			\draw[] (.75, 3.75) to (0,4);
			\draw[] (.75,3.75) to (1,3);
			\draw[->] (-.75, 3.75) to (-1.25, 4.5);
			\draw[->] (.75, 3.75) to (1.25, 4.5);
		\end{tikzpicture}
		=
		[2]\overline{\W}
		\]
		We note that $\deg(\W(s))=\deg(\overline{\W})+1$. 
		
		After resolving the square in $\W(s')$,  that resulted from the bubble of $\W(s)$ illustrated above,
		we obtain two webs  $\overline{\W}$ and $\overline{\W'}$, illustrated  left to right in the middle part of the equalities below.
		Given our earlier assumptions and Lemma \ref{bigonsquare}, the $OW$-move must create a square, say $T$, on $\overline{\W}'$
		as shown below. Resolving $T$ we write $\overline{\W}'=\overline{\W}_1+\overline{\W}_2$, as shown below.
				
		\[
		\W(s')=
		\begin{tikzpicture}[scale=.5]
			\draw[mid>] (0,-1) to (0,0);
			\draw (0,0) to (.25,.25);
			\draw[dotted] (.25,.25) to (.75,.75);
			\draw (.75,.75) to (1,1);
			\draw[mid<] (0,0) to (-1,1);
			\draw[mid>] (-1,1) .. controls (-1.5,2) .. (-1,3);
			\draw[->] (1,1) to (1.5,.5);
			\draw[mid<] (1,3) to (1.5,3.5);
			\draw[mid<] (0,4) to (0,5);
			\draw[mid>] (-1,1) to (0,1.5);
			\draw[mid>] (1,1) to (0,1.5);
			\draw[mid<] (0,1.5) to (0,2.5);
			\draw[mid>] (0,2.5) to (-1,3);
			\draw[mid>] (0,2.5) to (1,3);
			\draw[] (-1,3) to (-.75, 3.75);
			\draw[] (-.75, 3.75) to (0,4);
			\draw[] (.75, 3.75) to (0,4);
			\draw[] (.75,3.75) to (1,3);
			\draw[->] (-.75, 3.75) to (-1.25, 4.5);
			\draw[->] (.75, 3.75) to (1.25, 4.5);
		\end{tikzpicture}
		=~
		\begin{tikzpicture}[scale=.5]
			\draw[mid>] (0,-1) to (0,0);
			\draw (0,0) to (.25,.25);
			\draw[dotted] (.25,.25) to (.75,.75);
			\draw (.75,.75) to (1,1);
			\draw[mid>] (1,1) to (1,3);
			\draw[-<] (0,0) to (-1,1);
			\draw (-1,1) ..controls (-1.25,2).. (-1,3);
			\draw[->] (1,1) to (1.5,.5);
			\draw[mid<] (1,3) to (1.5,3.5);
			\draw[mid<] (0,4) to (0,5);
			\draw[<-] (-1,3) to (-.75, 3.75);
			\draw[] (-.75, 3.75) to (0,4);
			\draw[] (.75, 3.75) to (0,4);
			\draw[] (.75,3.75) to (1,3);
			\draw[->] (-.75, 3.75) to (-1.25, 4.5);
			\draw[->] (.75, 3.75) to (1.25, 4.5);
		\end{tikzpicture}
		+
		\begin{tikzpicture}[scale=.5]
			\draw[mid>] (0,-1) to (0,0);
			\draw (0,0) to (.25,.25);
			\draw[dotted] (.25,.25) to (.75,.75);
			\draw (.75,.75) to (1,1);
			\draw[->] (1,1) to (1.5,.5);
			\draw[mid<] (1,3) to (1.5,3.5);
			\draw[mid<] (0,4) to (0,5);
			\draw[<-] (-1,3) to (-.75, 3.75);
			\draw[] (-.75, 3.75) to (0,4);
			\draw[] (.75, 3.75) to (0,4);
			\draw[] (.75,3.75) to (1,3);
			\draw[->] (-.75, 3.75) to (-1.25, 4.5);
			\draw[->] (.75, 3.75) to (1.25, 4.5);
			\draw (1,3) .. controls (0,2).. (-1, 3);
			\draw[mid<] (0,0) arc(180:90:1);
		\end{tikzpicture}
		=
		\overline{\W}+
		\underbrace{\begin{tikzpicture}[scale=.5]
				\draw[mid>] (0,-1) to (0,0);
				\draw (0,0) to (.25,.25);
				\draw[dotted] (.25,.25) to (.75,.75);
				\draw (.75,.75) to (1,1);
				\draw[->] (1,1) to (1.5,.5);
				\draw[mid>] (1.5,3.5) .. controls (0,3).. (-1.25, 4.5) node[midway,below] (B) {} node[midway,above] (B1) {};
				\draw[mid>] (0,5) .. controls (0,4) and (.75, 3.75) .. (1.25, 4.5) node[midway,above] (C) {};
				\draw[mid<] (0,0) arc(180:90:1) node[midway,below] (A) {};
				\draw[dashed,bend left=20] (A) to (B1);
				\draw[dashed,bend right=20] (B) to (C);
		\end{tikzpicture}}_{\overline{\W}_1}
		~+~
		\underbrace{\begin{tikzpicture}[scale=.5]
				\draw[mid>] (0,-1) to (0,0);
				\draw (0,0) to (.25,.25);
				\draw[dotted] (.25,.25) to (.75,.75);
				\draw (.75,.75) to (1,1);
				\draw[->] (1,1) to (1.5,.5);
				\draw[mid>] (1.5,3.5) .. controls (1,3.25) and (.75, 4.25).. (1.25, 4.5) node[midway,above] (B) {};
				\draw[mid>] (0,5) .. controls (0, 4.5) and (-1,3.5) .. (-1.25, 4.5) node[midway,above] (C) {};
				\draw[mid<] (0,0) arc(180:90:1) node[pos=.3,below] (A) {}node[pos=.8,below] (A1) {};
				\draw[dashed, bend left=15] (A1) to (B);
				\draw[dashed, bend left=15] (A) to (C);
		\end{tikzpicture}}_{\overline{\W}_2}
		\]
		Observe that each of   $\overline{\W}_1$ and $\overline{\W}_2$ transforms into $\overline{\W}$ under two applications of the $OW$-move. 
		Each move will be performed in a disc neighborhood of one the four arcs indicated by dashed lines in the right hand side of above equation.

		Once again, since  $\overline{\W}$  has fewer than $n$ edges, by the induction hypothesis,  we have
		 $$\deg(\overline{\W}_1)=\deg(\overline{\W})+\{-2,0,2\}\ \ \  {\rm and }\ \ \ \deg(\overline{\W}_2)=\deg(\overline{\W})+\{-2,0,2\}.$$
		By the positivity of the coefficients of crossingless webs, we now have		
		\begin{align*}
			&\deg(\W(s'))=\max(\deg(\overline{\W}),\deg(\overline{\W}_1),\deg(\overline{\W}_2))\\
		&=\max(\deg(\overline{\W}),\deg(\overline{\W})+\{-2,0,2\})=\deg(\overline{\W})+\{0,2\}\,.
		\end{align*}		
	which gives $\deg(\W(s'))=\deg(\W(s))+\{-1,1\}$.\newline

		\noindent \textbf{Subcase(B4).} Suppose that $P$
		$m$-gon with $m>4$, no edges of which are edges to bubbles of $\W(s)$.
		Then $W(s)$ 
		has at least one square with one of its sides being a component of $E\cap \W(s)$. Then,
		
		\[
		\W(s)=\begin{tikzpicture}[scale=.4]
			\draw[mid>] (0,-1) to (0,0);
			\draw (0,0) to (.25,.25);
			\draw[dotted] (.25,.25) to (.75,.75);
			\draw (.75,.75) to (1,1);
			\draw[mid>] (1,1) to (1,3);
			\draw[] (1,3) to (0,4);
			\draw[] (0,4) to (-1,3);
			\draw[mid<] (0,0) to (-1,1);
			\draw[mid>] (-1,1) to (-1,3);
			\draw[mid>] (-1,1) to (-3,1);
			\draw[mid<] (-3,1) to (-3,3);
			\draw[mid>] (-3,3) to (-1,3);
			\draw[mid>] (-3.5,.5) to (-3,1);
			\draw[<-] (-3.5,3.5) to (-3,3);
			\draw[->] (1,1) to (1.5,.5);
			\draw[mid<] (1,3) to (1.5,3.5);
			\draw[->] (0,4) to (0,5);
		\end{tikzpicture}
		=
		\underbrace{\begin{tikzpicture}[scale=.4]
				\draw[mid>] (0,-1) to (0,0);
				\draw (0,0) to (.25,.25);
				\draw[dotted] (.25,.25) to (.75,.75);
				\draw (.75,.75) to (1,1);
				\draw[mid>] (1,1) to (1,3);
				\draw[] (1,3) to (0,4);
				\draw[->] (0,4) to (-1,3);
				\draw[-<] (0,0) to (-1,1);
 				\draw (-1,1) to[out=135,in=225] node[midway,right] (B) {} (-1,3) ;
				\draw (-2.5,1) to[out=45,in=-45] node[midway,left] (A) {} (-2.5,3);
				\draw[->] (-3,.5) to (-2.5,1);
				\draw[mid<] (-3,3.5) to (-2.5,3);
				\draw[->] (1,1) to (1.5,.5);
				\draw[mid<] (1,3) to (1.5,3.5);
				\draw[->] (0,4) to (0,5);
				\draw[dashed] (A) to (B);
		\end{tikzpicture}}_{\W_1}
		~+~
		\underbrace{\begin{tikzpicture}[scale=.4]
				\draw[mid>] (0,-1) to (0,0);
				\draw (0,0) to (.25,.25);
				\draw[dotted] (.25,.25) to (.75,.75);
				\draw (.75,.75) to (1,1);
				\draw[mid>] (1,1) to (1,3);
				\draw[] (1,3) to (0,4);
				\draw[->] (0,4) to (-1,3);
				\draw[-<] (0,0) to (-1,1);
				\draw (-1,1) to[out=135,in=45] (-2.5,1);
				\draw  (-1,3) to[out=225,in=-45] (-2.5,3);
				\draw[->] (-3,.5) to (-2.5,1);
				\draw[<-] (-3,3.5) to (-2.5,3);
				\draw[->] (1,1) to (1.5,.5);
				\draw[mid<] (1,3) to (1.5,3.5);
				\draw[->] (0,4) to (0,5);
		\end{tikzpicture}}_{\W_2}
		\]
		and similarly in  $E\cap \W'(s)$ we have,
		
	\begin{align*}
			\W(s')&=\begin{tikzpicture}[scale=.6]
				\draw[mid>] (0,-1) to (0,0);
				\draw (0,0) to (.25,.25);
				\draw[dotted] (.25,.25) to (.75,.75);
				\draw (.75,.75) to (1,1);
				\draw[mid<] (1,3) to (0,4);
				\draw[mid>] (0,4) to (-1,3);
				\draw[mid<] (0,0) to (-1,1);
				\draw[mid>] (-1,1) to (-2,1);
				\draw[mid<] (-2,1) to (-2,3);
				\draw[mid>] (-2,3) to (-1,3);
				\draw[mid>] (-1,1) to (0,1.5);
				\draw[mid>] (1,1) to (0,1.5);
				\draw[mid<] (0,1.5) to (0,2.5);
				\draw[mid>] (0,2.5) to (-1,3);
				\draw[mid>] (0,2.5) to (1,3);
				\draw[mid>] (-2.5,.5) to (-2,1);
				\draw[<-] (-2.5,3.5) to (-2,3);
				\draw[->] (1,1) to (1.5,.5);
				\draw[mid<] (1,3) to (1.5,3.5);
				\draw[->] (0,4) to (0,5);
			\end{tikzpicture}
			=
			\underbrace{\begin{tikzpicture}[scale=.6]
					\draw[mid>] (0,-1) to (0,0);
					\draw (0,0) to (.25,.25);
					\draw[dotted] (.25,.25) to (.75,.75);
					\draw (.75,.75) to (1,1);
					\draw[mid<] (0,0) to (-1,1);
					\draw[mid>] (-1,1) to (-2,1);
					\draw[mid<] (-2,1) to (-2,3);
					\draw[->] (-2,3) to (-1,3);
					\draw[mid>] (-1,1) to (0,1.5);
					\draw[mid>] (1,1) to (0,1.5);
					\draw[-<] (0,1.5) to (0,2.5);
					\draw (-1,3) to[out=0,in=-90] (0,4);
					\draw (1,3) to[out=225,in=90] (0,2.5);
					\draw[mid>] (-2.5,.5) to (-2,1);
					\draw[<-] (-2.5,3.5) to (-2,3);
					\draw[->] (1,1) to (1.5,.5);
					\draw[-<] (1,3) to (1.5,3.5);
					\draw[->] (0,4) to (0,5);
			\end{tikzpicture}}_{\W'_1}
			+
			\begin{tikzpicture}[scale=.5]
				\draw[mid>] (0,-1) to (0,0);
				\draw (0,0) to (.25,.25);
				\draw[dotted] (.25,.25) to (.75,.75);
				\draw (.75,.75) to (1,1);
				\draw[mid<] (0,0) to (-1,1);
				\draw[mid>] (-1,1) to (-2,1);
				\draw[mid<] (-2,1) to (-2,3);
				\draw[->] (-2,3) to (-1,3);
				\draw[mid>] (-1,1) to (0,1.5);
				\draw[mid>] (1,1) to (0,1.5);
				\draw[-<] (0,1.5) to (0,2.5);
				\draw (-1,3) to[out=0,in=90] (0,2.5);
				\draw (1,3) to[out=225,in=-90] (0,4);
				\draw[mid>] (-2.5,.5) to (-2,1);
				\draw[<-] (-2.5,3.5) to (-2,3);
				\draw[->] (1,1) to (1.5,.5);
				\draw[-<] (1,3) to (1.5,3.5);
				\draw[->] (0,4) to (0,5);
			\end{tikzpicture}
			\\[1em]&=
			\W_1'+
			\underbrace{\begin{tikzpicture}[scale=.5]
					\draw[mid>] (0,-1) to (0,0);
					\draw (0,0) to (.25,.25);
					\draw[dotted] (.25,.25) to (.75,.75);
					\draw (.75,.75) to (1,1);
					\draw[mid>] (-1,1) to[out=45, in=-45] (-1,3);
					\draw[->] (1,1) to (1.5,.5);
					\draw[mid<] (0,0) arc(180:90:1) node[midway,below] (B) {};
					\draw[mid>] (1.25,3.5) .. controls (.75,3) and (-.25,3.5) .. (-.25,4) node[midway,above] (A) {};
					\draw[dashed,bend right=5] (A) to (B);
			\end{tikzpicture}}_{\overline{\W}'_1}
			~+~
			\underbrace{\begin{tikzpicture}[scale=.5]
					\draw[mid>] (0,-1) to (0,0);
					\draw (0,0) to (.25,.25);
					\draw[dotted] (.25,.25) to (.75,.75);
					\draw (.75,.75) to (1,1);
					%
					\draw[<-]  (-.5,3) to[out=-45, in=90] node[pos=.3,below] (B) {} (1,1) ;
					\draw[->] (1,1) to (1.5,.5);
					\draw[mid>] (1.5,3.5) .. controls (1,3) and (0,3.5) .. (0,4) node[midway,above] (A) {};;
					\draw[mid<] (0,0) to (-1,.5);
					\draw[dashed] (B) to (A);
			\end{tikzpicture}}_{\overline{\W}'_2}\,.
		\end{align*}
		
We observe the following relations:
		\begin{align*}
			\W_1\xmapsto{\mbox{$OW$-move}} \W_1'\,,
			&&
			\overline{\W}'_1 \xmapsto{\mbox{$OW$-move}} \W_1\,,
			&&
			\overline{\W}'_2 \xmapsto{\mbox{$OW$-move}} \W_2\,.
		\end{align*}
		By the inductive hypothesis, these imply 
		\begin{align*}
			\deg(\W_1)=\deg(\W_1')+\{-1,1\}\,,
			&&
			\deg(\overline{\W}'_1)=\deg(\W_1)+\{-1,1\}\,,
			&&
			\deg(\overline{\W}'_2)=\deg(\W_2)+\{-1,1\}\,.
		\end{align*}
		Positivity of the leading coefficients implies
		\begin{align*}
			\deg(\W(s'))&=\max(\deg(\W_1'), \deg(\overline{\W}'_1), \deg(\overline{\W}'_2))
			\\&=\max(\deg(\W_1)+\{-1,1\},\deg(\W_2)+\{-1,1\})
			\\
			&
			=\max(\deg(\W_1),\deg(\W_2))+\{-1,1\} 
			\\&=
			\deg(\W(s))+\{-1,1\} \,.\qedhere
	\end{align*}

	\end{proof}

	\medskip
	\smallskip
	
	We use the notation for the degree function 
	\begin{align*}
		d(s)=\deg(Y(s))=-2(\a_+(s)-\a_-(s))-3(\b_+(s)-\b_-(s))+\deg(\lla \W(s)\rra)\,.
	\end{align*} 
	For positive diagrams this reduces to 
	\begin{align*}
		d(s)=\deg(Y(s))=-2\a(s)-3\b(s)+\deg(\lla \W(s)\rra)\,.
	\end{align*} 
	
		\begin{prop}\label{prop.dec}
		Let $D$ be a positive diagram. Suppose $s$ and $s'$ are two states of $D$ such that $s'$ is obtained from $s$ by an $OW$-move. Then
		\begin{align*}
			d(s')=d(s)&& \mbox{or} &&  d(s')=d(s)-2\,.
		\end{align*}
		In particular, $OW$-moves on $s$ do not increase the degree of $Y(s)$. 
	\end{prop}
	\begin{proof}
		The phases of states related by an $OW$-move differ by a factor of $-q^{-1}$:
		\[
		\phi(s')=(-1)^{\b(s)+1}q^{-2(\a(s)-1)-3(\b(s)+1)}=-q^{-1}\phi(s)\,.
		\]
		Therefore 
		\begin{align*}
			d(s')=\deg(\lla \W(s')\rra)+\deg(\phi(s))-1 && \mbox{and}&&
			d(s)=\deg(\lla \W(s)\rra)+\deg(\phi(s))\,.
		\end{align*}
		Now from Lemma \ref{lem.webdec} either $\deg(\lla \W(s')\rra )=\deg(\lla \W(s)\rra)+1$ or $\deg(\lla \W(s')\rra )=\deg(\lla \W(s)\rra )-1$. The claim readily follows.	
	\end{proof}
	
	\smallskip
	
	\begin{rem}
		Under single changes of resolutions, the degree of $Y(s)$ is preserved modulo 2. Also note that each reduced expression in Equation \eqref{eq.webrels} contributes a polynomial in $q$ of constant degree modulo 2. Since $d(O)=2(\scs-\xings)$ for positive diagrams, we have $\lla D\rra\in q^{2(\scs-\xings)}\bZ[q^{-2}]$. In particular, the $\Atwo$ invariant of a positive link is valued in $\bZ[q^{2},q^{-2}]$.
	\end{rem}
	\medskip
	
	
	\section{The two leading terms of the $\Atwo$ invariant for positive links }\label{sec:leadtwo}
	
	\subsection{The $O$ and $W$ state graphs} We consider graphs associated to a link diagram $D$ for various types of crossing resolutions. We have already introduced the oriented and web resolutions of crossings in Figure \ref{fig.res}. 
	
	To each state $s$ of a link diagram $D$, there is an associated \emph{state diagram} $D_s$ constructed as follows. First consider the all $O$-resolution of $D$ which consists of non-intersecting circles on  $S^2$. For each crossing of $D$ which is assigned a $W$-resolution in $s$, add an edge between two curves in the position of the crossing. We also have the \emph{state graph} $\G_s\coloneqq\G_s(D)$ associated to $s$, obtained from $D_s$ by collapsing each circle to a point. 
	By definition, for each state $s$, the graph $\G_s$ is a spanning subgraph of $\G_W$, i.e.\! it contains all the vertices of  $\G_W$.
	We denote the number of vertices in $\G_W$ (Seifert circles in $D$) by $\scs$ and the number of edges in $\G_W$ (crossings in $D$) by $\edges$.

		\begin{rem}\label{rem.same}
		Note that the $O$-resolution is orientation preserving, and the state diagram $D_O$ coincides with the set of Seifert circles  of $D$. In Seifert's algorithm for constructing a Seifert surface, twisted bands join two circles in the positions of each crossing. The core of each band corresponds an edge in the diagram $D_s$ and $D_W$ is exactly the Seifert diagram of $D$. In this way $\G_W$ the \emph{$W$-graph} is a spine of the Seifert surface for $D$. 
		The graph associated with the all-$O$ state $\G_O$, which may also call the \emph{$O$-graph}, has no edges.
	\end{rem}
	
	The \emph{reduced state graph $\G_s'\coloneqq\G_s'(D)$} for a state $s$ is obtained from $\G_s$ by removing all multiplicity edges. We call an edge of $\G'_W$ a \emph{reduced edge} and write $\edges'$ for the number of reduced edges in the graph.
	
Our main result in this section is the following.
	
	\begin{thm}\label{thm.second}
		Let $L$ be a link with a positive diagram $D$. The first two leading terms of the $\Atwo$ invariant are expressed as 
		\[\lla L\rra = q^{2(\scs-\xings)} +(\scs-\edges')q^{2(\scs-\xings-1)}+\textup{lower degree terms}\,,\]
		where $\scs, \xings$ and $\edges'$ are defined for $D$ as above.
	\end{thm}
	
Note that Theorem \ref{thm.second} implies that the quantities
$\scs-\xings$ and $\scs-\edges'$ are invariants of $L$ (i.e.\! independent of the choice of the positive diagram). 
		
		\begin{rem}\label{rem.deccond} Given the language of state graphs, we formulate another implication of Proposition \ref{prop.dec} and its proof that  will also be useful in computing coefficients of the $\Atwo$ invariant.
		In \hyperlink{lem.webdec.A}{{Case(A)}} of the proof of Lemma \ref{lem.webdec}, the clearing region is bounded by two polygons and in this case we showed $\deg(\lla \W(s')\rra)=\deg(\lla \W(s)\rra)-1$. Therefore $d(s')=d(s)-2$. Moreover, this $OW$-move decreases the number of components of both $\W(s)$ and the state graph $\G_s$. Consequently, a state $s$ corresponding to a (reduced) graph with $\scs-k$ components has degree at most $2(\scs-\xings-k)$.
	\end{rem}

	 \subsection{Proof of Theorem \ref{thm.second}}

	Recall that $Y(O)=q^{-2\xings}[3]^{\scs}$ and $d(O)=2(\scs-\xings)$.
	 Every state of $D$ is obtained from $O$ by replacing a number of $O$-resolutions by $W$-resolutions.
	Thus $s$ is obtained from $O$ by performing a number of $OW$-moves. 
	Now if $s_1$ is a state on a positive diagram $D$ that is obtained from $O$ by an $OW$-move, then in $s_1$ two Seifert circles of $O$ are merged as shown in Figure \ref{fig.OW}.
	By Remark \ref{rem.deccond}, $d(s_1)=d(O)-2$. If $s_2$ is obtained by performing additional $OW$-moves on $s_1$, then
	$d(s_2)\leq d(s_1)<d(O)$. 
	  It follows that $\deg(\lla D\rra)=2(\scs-\xings)$ for positive $D$ and the all-$O$ resolution is the unique state contributing to the highest coefficient of $\lla D\rra$. 
	  Thus, the leading coefficient of  $\lla D\rra$ is one.

	 To compute the second coefficient we need the following lemma.

	\begin{lem}\label{lem.oneW}
		Let $D$ be a positive diagram and let  $s\neq O$ be a state
		for which $d(s)=2(\scs-\xings-1)$. Then, all the $W$-resolutions in $s$ 
		occur between a single pair of Seifert circles, i.e.\! states $s$ where $\G_s'$ has a single edge. For such states
		\[
		Y(s)=(-q)^{-2\xings-\b(s)}[3]^{\scs-1}[2]^{\b(s)}\,.\]
	\end{lem}

\begin{proof}
		By assumption, $s$ has at least one $W$-resolution.  Therefore its state graph has at most $\scs-1$ components, and so, by Remark \ref{rem.deccond}, $d(s)\leq 2(\scs-\xings-1)$. Moreover, if $s'$ is another state with additional $W$-resolutions between a different pair of vertices, then this further decreases the number of components of the state graph. For such $s'$, $d(s')<d(s)$. Therefore, only states with a single edge in their reduced graph have degree $2(\scs-\xings-1)$. The evaluation of such a state is a straightforward computation.
	\end{proof}

	We now complete the proof of Theorem \ref{thm.second}.

			\begin{proof}Let $D=D(L)$ be a positive diagram of $L$. The states $s$ of $D$ which contribute to the leading terms come from the all-$O$ resolution of $D$ as well as those with a single reduced edge by Lemma \ref{lem.oneW}. As indicated above, 
		\[
		Y(O)=q^{-2\xings}[3]^\scs=q^{2(\scs-\xings)}+\scs q^{2(\scs-\xings-1)}+\textup{lower degree terms}\] and for states $s$ corresponding to a single reduced edge 
		\[Y(s)=(-1)^{\b(s)}q^{2(\scs-\xings-1)}+\textup{lower degree terms}\,.\] 
		It remains to compute the sum over these states weighted by $Y(s)$.
		
		Enumerate the edges of $\G_W'$ for $i=1,\dots, e'$ and let $k_i$ be the multiplicity of the $i$-th reduced edge relative to $\G_W$. Thus we obtain the coefficient
		\[
		\sum_{i=1}^{e'}\sum_{j=1}^{k_i} {k_i \choose j}(-1)^{j}=\sum_{i=1}^{e'}-(1-\delta_{k_i,0})=-e'\,. \qedhere
		\]
	\end{proof}
	\medskip

	\section{A characterization of positive fibered links}\label{sec:char} 
	
	In this section we provide a characterization of positive fibered  links in terms  of the second coefficient of  the quantum $\Atwo$ invariant.
	Alternatively, our characterization can be stated in terms of the Seifert graph $\G_W=\G_W(D)$ of any positive diagram of the link.
	
	For a positive link $L$ and any positive diagram $D=D(L)$, let  $\g_i$  denote the $i$-th coefficient of the quantum $\Atwo$ invariant, specifically the coefficient of $q^{2(\scs-\xings-i+1)}$. From Theorem \ref{thm.second}, 
	we have $\g_1 =1$ and $ \g_2 = (\scs-\edges')$.

	\begin{thm}\label{thm.fiber} Suppose that $L$ admits a connected positive  diagram $D$ with reduced Seifert graph $\G_W'$.
	The following are equivalent:
	\begin{enumerate}[(1)]
			\item $S^3\setminus L$ fibers over $S^1$,
			\item $\G_W'$ is a tree,
			\item $\gamma_2=1$.
		\end{enumerate}
	\end{thm}
	
Before we proceed with the proof we explain how Theorem \ref{thm.fiberi} is derived from Theorem \ref{thm.fiber}:
Suppose that $L$ is a positive non-split link. Any positive diagram of  $L$ must be non-split (i.e.\! connected). Hence,  $L$ is fibered if and only if 
$\g_2=1$ by Theorem \ref{thm.fiberi}.

	\medskip
	\smallskip

For the proof of Theorem	 \ref{thm.fiber}  we need to recall the $A$, $B$-resolutions of link diagrams appearing the definition of the classical Kauffman bracket,
which in turn is related to the Jones polynomial. In this setting, from a  crossing of an un-oriented connected link diagram $D$  we obtain the $A$-resolution and the $B$ resolution
as illustrated in Figure \ref{fig.resolve}. The A-resolution (resp.\! B-resolution) is defined by deleting the part of the diagram swept over by rotating the overcrossing clockwise (resp.\! counterclockwise).  Applying the $A$-resolution  (resp.\! $B$-resolution) to all the crossings of $D$ gives
 gives a collection of simple closed curves $v_A$ (resp.\! $v_B$) called state circles. The all-$A$ (resp.\! all-$B$) state graph of $D$, denoted by $\G_A\coloneqq \G_A(D)$ (resp.\! $\G_B\coloneqq \G_B(D)$)   has vertex 
 set $v_A$ (resp.\! $v_B$) and each edge connects points on state circles  
 where the crossing resolutions where performed. 
Associated to the all-$A$ and all-$B$ resolutions of $D$ one has the state surfaces $S_A(D)$ and $S_B(D)$: Each state circle in $v_A$ (resp.\! $v_B$) bounds an embedded disc on the projection plane.
The discs can be made disjoint by pushing their interiors below the projection plane. Then for each resolved crossing we join the corresponding arcs of the state circles by a half-twisted band as shown in Figure \ref{fig.resolve}.  The reader is referred to \cite{Gutsbook} and references therein for additional details. Clearly the graph $\G_A$ (resp.\! $\G_B$)
	is a spine of the surface $S_A(D)$ (resp.\! $S_B(D)$).
	
	\begin{figure}[h!]
		\includegraphics[scale=.6]{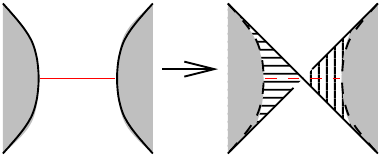}
		\hspace{2cm}
		\includegraphics[scale=.6]{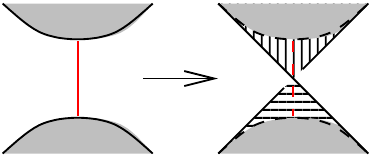}
		\caption{The $A$-resolution (left) the $B$-resolution (right)  of a crossing and their contribution to state surfaces. In both cases the edges of the state graph are shown in red.}
		\label{fig.resolve}
	\end{figure}
	
	\begin{defn}
	The diagram $D$ is called \emph{$A$-adequate}
	(resp.\!  \emph{$B$-adequate}) if  ${\mathbb G}_A$ (resp.\!  ${\mathbb G}_B$) contains no $1$-edge loops.
	\end{defn}
	
	Let $D$ be a connected positive diagram of a link $L$ and consider the Seifert graph ${\G}'_W$ obtained from the all-$W$ resolution described in the beginning of Section
	\ref{sec:leadtwo}.  We will use ${\overline D}$ to denote $D$ with the orientations ignored. Furthermore, for a crossing $x$ of $D$ we will use ${\overline x}$ to denote the corresponding crossing in ${\overline D}$.
	We can take the all-$B$ resolution of $\overline D$ and the corresponding state graph $\G'_B(\overline D)$ and the corresponding state surface 
	$S_B(\overline D)$. 	We need the following known lemma (proved for instance in \cite{Gutsbook}) and include a proof for completeness.

	\begin{lem}\label{lem.known} 
	 Let $D$ be a positive connected diagram of a link $L$. The following hold:
	\begin{enumerate}[(1)]
			\item the graphs $\G_W(D)$ and $\G_B({\overline D})$ are isomorphic,
			\item the state surface $S_B(\overline{D})$ is orientable and it is isotopic to the Seifert surface obtained by applying Seifert's algorithm to $D$,
			\item the diagram $\overline D$ is $B$-adequate.
		\end{enumerate}
\end{lem}
	
\begin{proof}  We begin with the following direct observation: For any crossing of $x$ of $D$, the $O$-resolution of $x$,  with orientations of the arcs ignored, is identical to the $B$-resolution of ${\overline x}$ in ${\overline D}$. See Figure \ref{fig.OB}. 
	
	\begin{figure}[h!]
		\begin{tikzpicture}[scale=1]
			\draw[->] (1,0) to [out=90,in=-90] (0,1);
			\draw[over,->] (0,0) to[out=90,in=-90] (1,1);
			\node at (.5,-.5) {positive crossing};
			\node at (.5,1.5) {\phantom{positive crossing}};
		\end{tikzpicture}
		$\xmapsto{\mbox{$O$-resolution}}$
		\begin{tikzpicture}
			\draw[->] (0,0) .. controls (.5,.25) and (.5,.75) .. (0,1);
			\draw[->] (1,0) .. controls (.5,.25) and (.5,.75) .. (1,1);
			\node at (.5,-.5) {$O$-resolution};
			\node at (.5,1.5) {\phantom{$O$-resolution}};
		\end{tikzpicture}
		$\xmapsto{\mbox{~$~\overline{~\cdot~}~$~}}$
		\begin{tikzpicture}
			\draw[] (0,0) .. controls (.5,.25) and (.5,.75) .. (0,1);
			\draw[] (1,0) .. controls (.5,.25) and (.5,.75) .. (1,1);
			\node at (.5,-.5) {$B$-resolution};
			\node at (.5,1.5) {\phantom{$B$-resolution}};
		\end{tikzpicture}
		\caption{The $O$-resolution of a positive crossing with the orientations ignored is identical to the $B$-resolution.}
		\label{fig.OB}
	\end{figure}
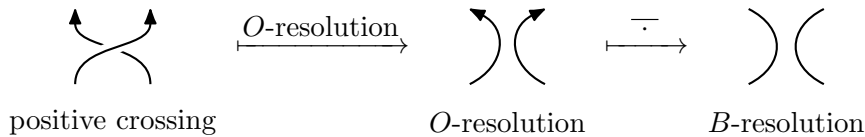
Now the first claim follows easily from the definitions of $\G_W(D)$ and $\G_B({\overline D})$
and Remark \ref{rem.same}.  It follows that the graph $G_B(\overline {D})$ is bipartite since $G_W(D)$ is. Hence ${S}_B({\overline D})$ is orientable (compare  \cite[Lemma 2.3]{Gutsbook}).
The remainder of the claim in (2) follows from part (1) and  the claim that $\overline D$ is $B$-adequate also follows from the fact that  $\G_B(\overline {D})$ is bipartite. 
\end{proof}

	We now give the proof of Theorem \ref{thm.fiber}.
	\begin{proof} 
		Suppose that $D=D(L)$ is a connected positive diagram of  fibered link $L$. By definition, $L$ is an oriented link. Applying Seifert's algorithm to $D$ we obtain a Seifert surface $S=S(D)$ that is of minimum genus \cite[Corollary 4.1]{homogeneous}.
		Then $S$ is a fiber of a fibration of $S^3\setminus L$ over $S^1$ \cite[Theorem 4.1.10]{AKbook}. On the other hand, by Lemma \ref{lem.known}, the surface $S$ is $S_B(\overline D)$. That is $S=S_B(D)$. Now \cite[Theorem 5.11]{Gutsbook} (its version for $B$-adequate diagrams)
states that the following two are equivalent. See also also \cite{Futer} for another proof of the result:
		
		\begin{enumerate}[(a)]
			\item $S$ is a fiber of a fibration $S^3\setminus L$ over $S^1$\,,
			\item the reduced graph ${\mathbb G}'_B(\overline D)$ is a tree.
		\end{enumerate}
		Since ${\mathbb G}_B(\overline D)=\G_W(D)$, we determine that $\G'_W(D)$ is a tree.
		Hence $ \g_2 = (\scs-\edges')=1$. This proves that $(1)\implies (2)\implies (3)$.

		To finish the proof we show that $(3)\implies (1)$.
		Suppose  $\gamma_2=(\scs-\edges')=1$. Then the reduced graph ${\mathbb G}'_W$ is a connected tree and again by
		\cite[Theorem 5.11]{Gutsbook}, the surface $S_B(\overline D)=S(D)$ is a fiber of a fibration $S^3\setminus L$ over $S^1$. Hence $L$ is fibered.
	\end{proof}

	\smallskip
	\medskip


	\section{The third coefficient of the $\Atwo$ invariant}\label{sec:leadthree}
	Continuing our work from Section \ref{sec:leadtwo} we compute the third coefficient of the $\Atwo$ invariant of positive links in terms of Seifert graphs of positive diagrams.	
	We prove the following.	
	
	\begin{thm}\label{thm.third}
		Let $L$ be a positive link. The first three leading terms of the $\Atwo$ invariant are expressed as 
		\[\lla L\rra = q^{2(\scs-\xings)} +(\scs-\edges')q^{2(\scs-\xings-1)}+
		\left({\scs-\edges'+1\choose 2}+\mu -\th \right)q^{2(\scs-\xings-2)} + \textup{lower degree terms}\]		
where $\scs$ and  $\xings$ are as in Theorem \ref{thm.second}, $\mu$ is the number of edges in $\G'_W$ that have multiplicity greater than one  in $\G_W$, and $\th$ is the number of pairs of edges in $\G'_W$  which are {mixed at a vertex} in  $\G'_W$.
\end{thm}

By Lemma \ref{lem.known} the Seifert graph $\G_W(D)$ of a positive diagram $D$ is the all-$B$ graph of  $D$ with orientations ignored.
The quantity $\theta$ for $G_B$ is defined in \cite{DL} where the authors compute the three trailing terms of the Jones polynomial of  $B$-adequate links. We recall the definition below.

Before we proceed with the proof of Theorem \ref{thm.third} we explain how Theorem \ref{thm.thirdi}, stated in the introduction, is a consequence.
\begin{proof}[Proof of Theorem \ref{thm.thirdi}]
Let $L$ be  a link with a connected positive  diagram  $D=D(L)$. Then
	\begin{enumerate}[(i)]
	\item the surface $S$ obtained by applying Seifert's algorithm to $D$ realizes $\chi(L)$,
	\item if the Seifert graph $\G_W(D)$ is a connected tree, then $D$ represents the unknot.
	\end{enumerate}
By Theorem \ref{thm.third}  the highest degree of $\lla L\rra $ is $2 \chi(\G_W(D))$.
By (i), and since $\G_W(D)$ is a spine for $S$,
$\chi(\G_W(D))=\chi(S)=\chi(L)\leq 1$.  If $\chi(L)=1$, then $\G_W(D)$ is  a tree and  by (ii) above $D$ represents the unknot.
Hence, part (1) of the theorem follows.

Again by	 Theorem \ref{thm.third}  the leading coefficient of $\lla L\rra $ is 1 and   $\g_2\coloneqq \scs-\edges'=\chi(\G'_W(D))\leq  1$,
proving part (2) of the theorem. Part (3) also follows at once from Theorem \ref{thm.third} since $\g_2=\scs-\edges'$.
\end{proof}

\begin{rem} 
The formulae for the coefficients $\g_i$  above are analogues of the Dasbach-Lin formulae \cite{DL} for the coefficients of the Jones polynomial of alternating, and more generally adequate, links.
The formulae obtained in \cite[Theorem 4.1]{DL} as well as the combinatorics underlying the proof their theorem are similar to these of Theorem \ref{thm.third}.
The actual values of the  coefficients of the two polynomials are different.
\end{rem}

\smallskip
\smallskip

	\subsection{Separated and mixed states} Here we define terminology and we prove some lemmas we need for the proof of Theorem \ref{thm.third}.
	
	Consider $\G_W=\G_W(D)$ as the Seifert graph of a positive diagram $D$ endowed with a cyclic ordering of half-edges at each vertex.
	\begin{defn}\label{def:mixed} Two edges  $e_1'$ and $e_2'$ of  $\G_W'$ are \emph{disjoint} if they do not share a common vertex. 
	If $e_1'$ and $e_2'$ are not disjoint, then we say they are \emph{separated} if, in the cyclic order at their common vertex in $\G_W$, the edges can be separated into two consecutive sets of edges, one over $e_1'$ and the other over $e_2'$ (ignoring other reduced edges in the graph). 
	If these edges are neither disjoint nor separated, then they are called \emph{mixed}. See Figure \ref{fig.alt}.
	\end{defn}

	Suppose $e_1'$ and $e_2'$ have a common vertex in $\G'_W$. The lifts of these edges in $\G_W$ are partitioned into sets of $(a_1, a_2,\dots , a_m)$ and $(b_1, b_2,\dots , b_m)$ of  parallel multiple edges, so that all the half-edges from $e_1$ and $e_2$ at the common vertex alternate in sets of size $a_1, b_1, a_1,b_2,\dots, a_m,b_m$. Note that the set of edges $a_1, b_1, a_1,b_2,\dots, a_m,b_m, a_{m+1}$ is equivalent to a set of edges $a_1+a_{m+1},b_2,\dots, a_m,b_m$ as shown in the Figure \ref{fig.alt}. We call $m$ the \emph{mixing index} of the pair of edges. Now $e_1'$ and $e_2'$ are separated if and only if $m=1$ and they are mixed otherwise. 

	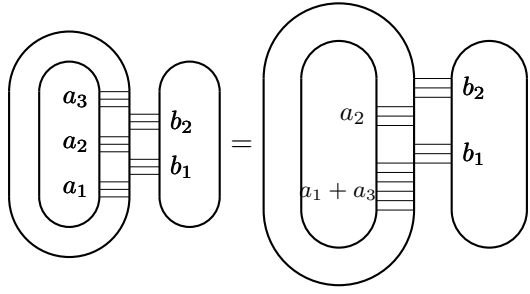
\begin{figure}[h!]
		\begin{tikzpicture}[scale=.4]
			\def\y{3.5}
			\draw (0,0) arc(0:-180:1);
			\draw (0,\y) arc(0:180:1);
			\draw (0,0) to (0,\y);
			\draw (-2,0) to (-2,\y);
			\draw (1,0) arc(0:-180:2);
			\draw (1,\y) arc(0:180:2);
			\draw (1,0) to (1,\y);
			\draw (-3,0) to (-3,\y);
			\draw (2,0) arc(180:360:1);
			\draw (2,\y) arc(180:0:1);
			\draw (2,0) to (2,\y);
			
			\draw (4,0) to (4,\y);
			\foreach \y / \ytext in {0/1,1.5/2,3/3}{
				\foreach \x in {0,.25,.5}{
				\draw[very thin] (0,\y+\x) to (1,\y+\x);
				\node[left] at (0,\y+.25) {\small$a_\ytext$};
				}
			}
			\foreach \y/ \ytext in {.75/1,2.25/2}{
				\foreach \x in {0,.25,.5}{
					\draw[very thin] (1,\y+\x) to (2,\y+\x);
					\node[right] at (2,\y+.25) {\small$b_\ytext$};
				}
			}
		\end{tikzpicture}
		~=~
		\begin{tikzpicture}[scale=.5]
			\def\y{3.5}
			\draw (0,0) arc(0:-180:1);
			\draw (0,\y) arc(0:180:1);
			\draw (0,0) to (0,\y);
			\draw (-2,0) to (-2,\y);
			\draw (1,0) arc(0:-180:2);
			\draw (1,\y) arc(0:180:2);
			\draw (1,0) to (1,\y);
			\draw (-3,0) to (-3,\y);
			\draw (2,0) arc(180:360:1);
			\draw (2,\y) arc(180:0:1);
			\draw (2,0) to (2,\y);
			\draw (4,0) to (4,\y);
			\foreach \y / \ytext in {0,.75,2.25/2}{
				\foreach \x in {0,.25,.5}{
					\draw[very thin] (0,\y+\x) to (1,\y+\x);					
				}
			}
			\node[left] at (0,2.5) {\small$a_2$};
			\node[left] at (0.3,.5) {\footnotesize$a_1+a_3$};
			\foreach \y/ \ytext in {1.25/1,3/2}{
				\foreach \x in {0,.25,.5}{
					\draw[very thin] (1,\y+\x) to (2,\y+\x);
					\node[right] at (2,\y+.25) {\small$b_\ytext$};
				}
			}
		\end{tikzpicture}
		\caption{An example of a state diagram over a mixed pair of edges with index 2. These alternating edges realize the equality $\G(a_1,b_1,a_2,b_2,a_3)=\G(a_1+a_3,b_1,a_2,b_2)$.}\label{fig.alt}
	\end{figure}

	The terminology above adapts to the components of reduced graphs $\G_s'$ which have exactly two edges, i.e.\! we ignore other vertices and edges away from a shared vertex. Suppose that $\G_s'$ has exactly one component with two edges, then we identify $\G_s\setminus\{v\mid \deg(v)=0\}$ with $\G(u_1,v_1,\dots, u_m,v_m)$ and the corresponding web is denoted $\W(u_1,v_1,\dots, u_m,v_m)$ where $0\leq u_j\leq a_j$ and $0\leq v_j\leq b_j$.

	 Since some $u_j$ and $v_j$ may be zero, the edges of $\G_s'$ may be mixed but the \emph{state mixing index} $n$ may be less than the mixing index $m$ of $\G_W'$. In this case $\G(u_1,v_1,\dots, u_m,v_m)=\G(u_1',v_1',\dots, u_n',v_n')$ with $u_j',v_j'>0$. In particular, if $n=1$ we call the state $s$ \emph{semi-mixed}.
	
		In a semi-mixed state, there is a set of consecutive indices $i,i+1,\dots i+\ell$ considered modulo $m$ such that in the states with graph $\G(u_1,v_1,\dots, u_m,v_m)\subseteq \G(a_1,b_1,\dots, a_m,b_m)$
all of the following conditions hold:
		\begin{enumerate}[(i)]
			\item the entries $u_j$ for $j<i$ and $j>i+\ell$ must be zero,
			\item at least one $u_i,\dots,u_{i+\ell} $ is nonzero,
			\item the entries $v_i,\dots, v_{i+\ell-1}$ must be zero,
			\item at least one $v_j$ is nonzero for  $j<i$ or $j\geq i+\ell$.
		\end{enumerate}
		We denote the support of the $u_j$ by $I_u\subset \{i,i+1,\dots i+\ell\}$ and the support of the $v_j$ by $I_v$. Observe that $$|I_u|+|I_v|\leq m+1 \ \ {\rm  and} \ \  |I_u\cap I_v|\leq 1.$$

In the remainder of this section, we prove some lemmas  that we will need for the proof of Theorem \ref{thm.third}.
The next lemma computes
 the contribution to the third coefficient of the web corresponding to a graph $\G(u_1,v_1,\dots, u_m,v_m)$ for a semi-mixed state, while the following two lemmas
 show that the number of states with a support $(I_u,I_v)$ is only a function of $|I_u|+|I_v|$. 		
	\begin{lem}\label{lem.chunksum}
		Consider the set of states whose web-resolution lies over a single pair of mixed reduced edges with state mixing index $n=1$ and a fixed set of supports $I_u$ and $I_v$. The leading term in the sum of $Y(s)$ over all such states is $(-1)^{|I_u|+| I_v|}q^{2(\scs-\edges-2)}$.
	\end{lem}
	\begin{proof}
		In any such state $Y(u_1,v_1,\dots, u_m,v_m)$ is given by $(-q)^{2\edges-\sum u_j - \sum  v_j}[3]^{\scs-2}\cdot[2]^{\sum u_j + \sum v_j}$
		and its leading term is $(-1)^{\sum u_j + \sum v_j}q^{2(\scs-\edges-2)}$. The sum over all such states with this support is
		\[
		\left(\prod_{i\in I_u}\sum_{u=1}^{a_i} {a_i\choose u}(-1)^u\right)
		\left(\prod_{i\in I_v}\sum_{v=1}^{b_i} {b_i\choose v}(-1)^v\right)
		=
		\left(\prod_{i\in I_u}(-1)\right)
		\left(\prod_{i\in I_v}(-1)\right)
		=
		(-1)^{|I_u|+|I_v|}\qedhere
		\]
	\end{proof}
	
	Having summed over states with a fixed support, it remains to count the number of support sets with a given size. We define
	\[
		C_{u,v}=\left|\{I_u,I_v\subseteq\{1,\dots, m\}\mid \mbox{$I_u,I_v$ determine semi-mixed states with $|I_u|=u$, $|I_v|=v$} \}\right|\,.
	\] 
	For example, there are $m{m\choose v}$ different support sets which have $|I_u|=1$ and $|I_v|=v$. Note that $C_{u,v}=C_{v,u}$ and $C_{u,v}=0$ if $u+v>m+1$. By rotational symmetry $C$ is always a multiple of $m$. For $|I_u|<m$, one can see this by assuming the last entry of $I_u$ as a set of consecutive integers is $1$ rather than $j$ with $1\leq j \leq m$. If $|I_u|=m$, then it is easy to see that
	$C_{m,v}=m\delta_{v,1}$. Therefore, we will also write $\overline{C}_{u,v}=C_{u,v}/m$.
	
	\begin{lem}\label{lem.CFormula}
		Consider a pair of mixed reduced edges with index $m$. For $u\leq v$, 
		\[
		\overline{C}_{u,v}=
		\begin{cases}
			\displaystyle{m\choose v} & \mbox{if $u=1$}
			\\
			\displaystyle\sum_{i=v+1}^{m-u+2} {i-1\choose v}{m-i\choose u-2} & \mbox{if $u>1$}
		\end{cases}\,.
		\]
	\end{lem}
	\begin{proof}
		In the case $u=1$ we assume that $I_u=\{1\}$. Then any set $I_v\subset\{1,\dots,m\}$ produces a separated graph and there are ${m\choose v}$ such sets of size $v$. 
		
		For $u>1$, we assume that as sets of consecutive integers modulo $m$ the last entry of $I_u$ is $1$. Then the extent of $I_v$ is determined by the first entry of $I_u$, which is therefore at least $v+1$. The index $i$ of the first entry of $I_u$ may vary from $v+1$ to $m-u+2$ and in each case there are ${i-1\choose v}{m-i\choose u-2}$ choices of sets. Summing over these possibilities gives the indicated expression.
	\end{proof}

Note that if $u>1$ and $v=1$, then $\sum_{i=v+1}^{m-u+2} {i-1\choose v}{m-i\choose u-2}={m\choose v}$. This observation is consistent with the symmetry $C_{u,v}=C_{v,u}$. More generally, $C$ has the following property.

		\begin{lem}\label{lem.CSymm}
			For all $u,v>1$, $C_{u,v}=C_{u+1,v-1}$.
		\end{lem}
		\begin{proof}
			A straightforward computation shows
			\begin{align*}
				\overline{C}_{u+1,v-1}-\overline{C}_{u,v}
				&=
				\sum_{i=v}^{m-u+1} {i-1\choose v-1}{m-i \choose u-1} 
				-
				\sum_{i=v+1}^{m-u+2} {i-1\choose v}{m-i \choose u-2}
				\\
				&=
				\sum_{i=v+1}^{m-u+2} {i-2\choose v-1}{m-i+1 \choose u-1} - {i-1\choose v}{m-i \choose u-2}
				\\
				&=
				\sum_{i=v+1}^{m-u+2} \left\lbrace{i-1\choose v}-{i-2\choose v}\right\rbrace\left\lbrace{m-i \choose u-1}+{m-i \choose u-2}\right\rbrace - {i-1\choose v}{m-i \choose u-2}
				\\
				&=
				\sum_{i=v+1}^{m-u+2} {i-1\choose v}{m-i \choose u-1}-{i-2\choose v}{m-i+1 \choose u-1}
				\\
				&=
				\sum_{i=v+1}^{m-u+1} {i-1\choose v}{m-i \choose u-1} - \sum_{i=v+2}^{m-u+2}{i-2\choose v}{m-i+1 \choose u-1}
				\\
				&=
				\sum_{i=v+2}^{m-u+2}{i-2\choose v}{m-i+1 \choose u-1} - {i-2\choose v}{m-i+1 \choose u-1}=0\,.\qedhere
			\end{align*}
		\end{proof}
	
	The last lemma in this subsection shows that the total contribution to the third coefficient
	of  semi-mixed states between  a pair of mixed reduced edges  is zero.

	\begin{lem}\label{lem.semisum}
		The sum of leading terms of all semi-mixed states over a pair of mixed reduced edges is zero.
	\end{lem}
	\begin{proof}
		We showed in Lemma \ref{lem.chunksum} that the sum over all states with a fixed support of size $\ell=u+v$ contributes $(-1)^{\ell}q^{2(\scs-\xings-2)}$. By Lemma \ref{lem.CSymm}, there are $\sum_{i=1}^{\ell-1}C_{i,\ell-i}=(\ell-1)C_{1,\ell-1}=m(\ell-1){m\choose \ell-1}$  supports of size $\ell$. As $\ell$ varies from $2$ to $m+1$ we compute
		\begin{align*}
			m\sum_{\ell=2}^{m+1} {m\choose \ell-1}(-1)^{\ell}(\ell-1)
			=
			m\sum_{\ell=1}^{m+1} {m\choose \ell-1}(-1)^{\ell}(\ell-1)
			=
			-m\sum_{\ell=0}^{m}{m\choose \ell} (-1)^{\ell}\ell = m\cdot \delta_{m,1}\,.
		\end{align*}
		At a mixed pair of edges $m>1$ and so the above sum evaluates to zero.
	\end{proof}
	
	\smallskip
	\smallskip
	
	\subsection{Some web identities} Here we establish a particular web identity (Lemma \ref{lem.TraceSquare})
	that we need for the proof of Theorem \ref{thm.third}.  We need two lemmas that discuss how to simplify webs that contain a sequence of squares.

   \begin{lem}\label{lem.CapSquare}
     The half-capped sequence of $k\geq 0$ squares simplifies as follows:
     \[
		\begin{tikzpicture}[scale=.78]
			\draw[mid<] (0,0) to (1,0);
			\draw[mid>] (1,0) to (2,0);
			\draw[mid<] (2,0) to (3,0);
			\draw[mid>] (3,0) to (4,0);
			\draw[mid>] (0,1) to (1,1);
			\draw[mid<] (1,1) to (2,1);
			\draw[mid>] (2,1) to (3,1);
			\draw[mid<] (3,1) to (4,1);
			\draw[mid<] (0,0) to (0,1);
			\draw[mid>] (1,0) to (1,1);
			\draw[mid<] (2,0) to (2,1);
			\draw[mid>] (3,0) to (3,1);
			\draw[mid<] (4,0) to (4,1);
			\draw[mid<] (0,0) to (-.5,-.5);
			\draw[mid>] (0,1) to (-.5,1.5);
			\draw[mid>] (4,1) arc(90:-90:.5);
			\node at (.5,.5) {$1$};
			\node at  (1.5,.5) {$2$};
			\node at  (2.5,.5) {$\dots$};
			\node at (3.5,.5) {$k$};
		\end{tikzpicture}
		=
		[2]^{k+1}~
		\begin{tikzpicture}
			\draw[->] (0,-.5) to (0,1.5);
		\end{tikzpicture}
	\]
	\end{lem}
	\begin{proof}
		We give a proof by induction. The base case $k=0$ is given by the bubble relation in Equation \eqref{eq.webrels}. If the claim holds for some $n$, then again using the bubble move we have 		\[
		\scalebox{.9}{
		\begin{tikzpicture}
			\draw[mid<] (0,0) to (1,0);
			\draw[mid>] (1,0) to (2,0);
			\draw[mid<] (2,0) to (3,0);
			\draw[mid>] (3,0) to (4,0);
			\draw[mid>] (0,1) to (1,1);
			\draw[mid<] (1,1) to (2,1);
			\draw[mid>] (2,1) to (3,1);
			\draw[mid<] (3,1) to (4,1);
			\draw[mid<] (0,0) to (0,1);
			\draw[mid>] (1,0) to (1,1);
			\draw[mid<] (2,0) to (2,1);
			\draw[mid>] (3,0) to (3,1);
			\draw[mid<] (4,0) to (4,1);
			\draw[mid<] (0,0) to (-.5,-.5);
			\draw[mid>] (0,1) to (-.5,1.5);
			\draw[mid>] (4,1) arc(90:-90:.5);
			\node at (.5,.5) {$1$};
			\node at  (1.5,.5) {$2$};
			\node at  (2.5,.5) {$\dots$};
			\node at (3.5,.5) {$n+1$};
		\end{tikzpicture}
	}
		=
		[2]\scalebox{.9}{\begin{tikzpicture}
			\draw[mid<] (0,0) to (1,0);
			\draw[mid>] (1,0) to (2,0);
			\draw[mid<] (2,0) to (3,0);
			\draw[mid>] (3,0) to (4,0);
			\draw[mid>] (0,1) to (1,1);
			\draw[mid<] (1,1) to (2,1);
			\draw[mid>] (2,1) to (3,1);
			\draw[mid<] (3,1) to (4,1);
			\draw[mid<] (0,0) to (0,1);
			\draw[mid>] (1,0) to (1,1);
			\draw[mid<] (2,0) to (2,1);
			\draw[mid>] (3,0) to (3,1);
			\draw[mid<] (4,0) to (4,1);
			\draw[mid<] (0,0) to (-.5,-.5);
			\draw[mid>] (0,1) to (-.5,1.5);
			\draw[mid>] (4,1) arc(90:-90:.5);
			\node at (.5,.5) {$1$};
			\node at  (1.5,.5) {$2$};
			\node at  (2.5,.5) {$\dots$};
			\node at (3.5,.5) {$n$};
		\end{tikzpicture}
	}
		=
		[2]^{n+2}~
		\begin{tikzpicture}[scale=.78]
			\draw[->] (0,-.5) to (0,1.5);
		\end{tikzpicture}\,.
	\]
	\end{proof}

Our second lemma shows how to simplify webs that contain  ``uncapped"  sequences of squares.
	\begin{lem}\label{lem.UncapSquare}
     The uncapped sequence of $k\geq0$ squares simplifies as follows:
     \[
     \scalebox{.9}{$
     \begin{tikzpicture}
			\draw[mid<] (0,0) to (1,0);
			\draw[mid>] (1,0) to (2,0);
			\draw[mid<] (2,0) to (3,0);
			\draw[mid>] (3,0) to (4,0);
			\draw[mid>] (0,1) to (1,1);
			\draw[mid<] (1,1) to (2,1);
			\draw[mid>] (2,1) to (3,1);
			\draw[mid<] (3,1) to (4,1);
			\draw[mid<] (0,0) to (0,1);
			\draw[mid>] (1,0) to (1,1);
			\draw[mid<] (2,0) to (2,1);
			\draw[mid>] (3,0) to (3,1);
			\draw[mid<] (4,0) to (4,1);
			\draw[mid<] (0,0) to (-.5,-.5);
			\draw[mid>] (0,1) to (-.5,1.5);
			\draw[mid>] (4,1) to (4.5,1.5);
			\draw[mid<] (4,0) to (4.5,-.5);
			\node at (.5,.5) {$1$};
			\node at  (1.5,.5) {$2$};
			\node at  (2.5,.5) {$\dots$};
			\node at (3.5,.5) {$k$};
	\end{tikzpicture}
		=
		\begin{cases}
			\begin{tikzpicture}
				\draw[mid>] (0,0) to (.5,.25);
				\draw[mid>] (1,0) to (.5,.25);
				\draw[mid<] (.5,.25) to (.5,.75);
				\draw[->] (.5,.75) to (0,1);
				\draw[->] (.5,.75) to (1,1);
			\end{tikzpicture}+
			\displaystyle
			\sum_{i=1}^{k/2} [2]^{2i-1}
			\begin{tikzpicture}
				\draw[->] (0,0) .. controls (.5,.25) and (.5,.75) .. (0,1);
				\draw[->] (1,0) .. controls (.5,.25) and (.5,.75) .. (1,1);
			\end{tikzpicture}
			&\mbox{if $k$ is even}
			\\[2em]
			\begin{tikzpicture}[scale=.5]
					\draw[<-] (-.5,1.5) .. controls (0, .5) and (1,.5) ..  (1.5,1.5);
					\draw[->] (-.5,-.5) .. controls (0,.5) and (1,.5) ..  (1.5,-.5);
			\end{tikzpicture} +
			\displaystyle
			\sum_{i=1}^{(k+1)/2} [2]^{2i-2}
			\begin{tikzpicture}
					\draw[->] (0,0) .. controls (.5,.25) and (.5,.75) .. (0,1);
					\draw[<-] (1,0) .. controls (.5,.25) and (.5,.75) .. (1,1);
			\end{tikzpicture}
			&\mbox{if $k$ is odd}
		\end{cases}
		$
	}
		\,.
		\]
	\end{lem}
	\begin{proof}
		We give an inductive proof joined across both cases. The $k=0$ base case is obvious and the $k=1$ base case follows from the square move in Equation \eqref{eq.webrels}. If the claim holds for some $n\geq0$,
		\[
		\scalebox{.9}{
		\begin{tikzpicture}[]
			\draw[mid<] (0,0) to (1,0);
			\draw[mid>] (1,0) to (2,0);
			\draw[mid<] (2,0) to (3,0);
			\draw[mid>] (3,0) to (4,0);
			\draw[mid>] (0,1) to (1,1);
			\draw[mid<] (1,1) to (2,1);
			\draw[mid>] (2,1) to (3,1);
			\draw[mid<] (3,1) to (4,1);
			\draw[mid<] (0,0) to (0,1);
			\draw[mid>] (1,0) to (1,1);
			\draw[mid<] (2,0) to (2,1);
			\draw[mid>] (3,0) to (3,1);
			\draw[mid<] (4,0) to (4,1);
			\draw[mid<] (0,0) to (-.5,-.5);
			\draw[mid>] (0,1) to (-.5,1.5);
			\draw[mid>] (4,1) to (4.5,1.5);
			\draw[mid<] (4,0) to (4.5,-.5);
			\node at (.5,.5) {$1$};
			\node at  (1.5,.5) {$2$};
			\node at  (2.5,.5) {$\dots$};
			\node at (3.5,.5) {$n+1$};
		\end{tikzpicture}
		=
		\begin{tikzpicture}[]
			\draw[mid<] (0,0) to (1,0);
			\draw[mid>] (1,0) to (2,0);
			\draw[mid<] (2,0) to (3,0);
			\draw[mid>] (3,0) to (4,0);
			\draw[mid>] (0,1) to (1,1);
			\draw[mid<] (1,1) to (2,1);
			\draw[mid>] (2,1) to (3,1);
			\draw[mid<] (3,1) to (4,1);
			\draw[mid<] (0,0) to (0,1);
			\draw[mid>] (1,0) to (1,1);
			\draw[mid<] (2,0) to (2,1);
			\draw[mid>] (3,0) to (3,1);
			\draw[mid<] (4,0) to (4,1);
			\draw[mid<] (0,0) to (-.5,-.5);
			\draw[mid>] (0,1) to (-.5,1.5);
			\draw[mid>] (4,1) arc(90:-90:.5);
			\node at (.5,.5) {$1$};
			\node at  (1.5,.5) {$2$};
			\node at  (2.5,.5) {$\dots$};
			\node at (3.5,.5) {$n-1$};
			\draw[->] (5,-.5) .. controls (4.75,.25) and (4.75,.75) .. (5,1.5);
		\end{tikzpicture}
		~+~
		\begin{tikzpicture}[]
			\draw[mid<] (0,0) to (1,0);
			\draw[mid>] (1,0) to (2,0);
			\draw[mid<] (2,0) to (3,0);
			\draw[mid>] (3,0) to (4,0);
			\draw[mid>] (0,1) to (1,1);
			\draw[mid<] (1,1) to (2,1);
			\draw[mid>] (2,1) to (3,1);
			\draw[mid<] (3,1) to (4,1);
			\draw[mid<] (0,0) to (0,1);
			\draw[mid>] (1,0) to (1,1);
			\draw[mid<] (2,0) to (2,1);
			\draw[mid>] (3,0) to (3,1);
			\draw[mid<] (4,0) to (4,1);
			\draw[mid<] (0,0) to (-.5,-.5);
			\draw[mid>] (0,1) to (-.5,1.5);
			\draw[mid>] (4,1) to (4.5,1.5);
			\draw[mid<] (4,0) to (4.5,-.5);
			\node at (.5,.5) {$1$};
			\node at  (1.5,.5) {$2$};
			\node at  (2.5,.5) {$\dots$};
			\node at (3.5,.5) {$n-1$};
		\end{tikzpicture}
	}
		\,.
		\]
		The claim follows from applying Lemma \ref{lem.CapSquare} and the inductive hypothesis to the above terms.
	\end{proof}
\medskip

Now, using  Lemma \ref{lem.UncapSquare}, we have the following:

	\begin{lem}\label{lem.TraceSquare}
		The following equality of webs holds for even $k\geq 0$:
		\[
			\left\la\!\!\left\la~\begin{tikzpicture}[scale=.74]
			\draw[mid<] (0,0) to (1,0);
			\draw[mid>] (1,0) to (2,0);
			\draw[mid<] (2,0) to (3,0);
			\draw[mid>] (3,0) to (4,0);
			\draw[mid>] (0,1) to (1,1);
			\draw[mid<] (1,1) to (2,1);
			\draw[mid>] (2,1) to (3,1);
			\draw[mid<] (3,1) to (4,1);
			\draw[mid<] (0,0) to (0,1);
			\draw[mid>] (1,0) to (1,1);
			\draw[mid<] (2,0) to (2,1);
			\draw[mid>] (3,0) to (3,1);
			\draw[mid<] (4,0) to (4,1);
			\draw[mid<] (0,0) to (-.5,-0);
			\draw[->] (-.5,-0) to[out=135,in=0] (-1.5,.5);
			\draw[->] (-.5,-0) to[out=-135,in=0] (-1.5,-.5);
			\draw[->] (0,1) to (-1.5,1);
			\draw[mid>] (4,1) to (5,1);
			\draw (5,1) to[out=45,in=180] (5.25,1.25);
			\draw (5,1) to[out=-45,in=180] (5.5,.5);
			\draw[mid<] (4,0) to (5.5,0);
			\draw (-1.5,1) arc(270:90:.5);
			\draw (-1.5,.5) arc(270:90:1);
			\draw (-1.5,-.5) arc(270:90:1.75);
			\draw[<-] (5.25,1.25) arc(-90:90:.375);
			\draw[<-] (5.5,.5) arc(-90:90:1);
			\draw (5.5,0) arc(-90:90:1.5);
			\draw (-1.5,2) to (5.25,2);
			\draw (-1.5, 2.5) to (5.5, 2.5);
			\draw (-1.5,3) to (5.5,3);
			\node at (.5,.5) {$1$};
			\node at  (1.5,.5) {$2$};
			\node at  (2.5,.5) {$\dots$};
			\node at (3.5,.5) {$k$};
	\end{tikzpicture}~\right\ra\!\!\right\ra
	~= [2]^2[3]+\sum_{i=1}^{k/2} [2]^{2i}[3]\,.
		\]
	\end{lem}
	\begin{proof}
		The identity follows immediately from Lemma \ref{lem.UncapSquare} by taking the appropriate closure and applying Equation \eqref{eq.webrels}.
	\end{proof}

	\smallskip
	\smallskip
	
	\subsection{The proof of Theorem \ref{thm.third}}
	
	We are now ready to give the proof of Theorem \ref{thm.third}.
	
	\begin{proof}
			Let $D=D(L)$ be a positive diagram of $L$. To compute the third coefficient of the invariant, we determine which states contribute to the degree $2(\scs-\xings-2)$ term. Remark \ref{rem.deccond} tells us that the state graphs contributing to the term of degree $q^{2(\scs-\xings-2)}$ have at least $\scs-2$ components. Those with exactly $\scs$ or $\scs-1$ components have exactly zero or one reduced edge, respectively. A state graph with $\scs-2$ components has exactly two reduced edges, since triangles are not allowed due to biparticity of the state graph. For states whose reduced graph has $\scs-\xings-2$ edges, we separate into cases where the edges are mixed at a vertex or not. 
			\\
		
		\noindent\textbf{Case(A).} From the all-$O$ state
		\[
		Y(O)=q^{-2\xings}[3]^\scs=q^{2(\scs-\xings)}+\scs q^{2(\scs-\xings-1)}+{\scs\choose2}_{\!2}q^{2(\scs-\xings-2)}+\textup{lower degree terms}
		\]
		where ${\scs\choose2}_2$ is a trinomial coefficient equal to ${\scs\choose 1}+{\scs\choose 2}={\scs+1\choose2}$.\\
		
		\noindent\textbf{Case(B).} If $s$ is a state with a single edge in $\G_s'$, then by Lemma \ref{lem.oneW}
		\[
		Y(s)=(-q)^{-2\xings-\b(s)}[3]^{\scs-1}[2]^{\b(s)}
		=(-1)^{\b(s)}q^{2(\scs-\xings-1)}\left(1+(v-1+\b(s))q^{-2}+\textup{lower degree terms}\right)\,.
		\]
		We sum the $q^{2(\scs-\xings-2)}$ term over all such states. As in the proof of Theorem \ref{thm.second}, we enumerate the edges of $\G_W'$ for $i=1,\dots, \edges'$ and let $k_i$ be the multiplicity of the $i$-th reduced edge relative of $\G_W$. We compute
		\[
		\sum_{i=1}^{\edges'}\sum_{j=1}^{k_i} {k_i\choose j}(-1)^j(\scs-1+j)=-\sum_{i=1}^{\edges'} (\scs-1+\delta_{k_i,1})=-(\scs-1)\edges'-|\lifts_1|=-\scs\edges'+\mu\,,
		\]
		where $|\lifts_1|$ denotes the number of edges in $\G'_W$ that appear with multiplicity one in the unreduced Seifert graph $\G_W$. Now recall that we denoted by $\mu$ the
		 edges in $\G'_W$ that appear with multiplicity more than one in $\G_W$. Since $\mu=\edges'-|\lifts_1|$,  the last equation follows.
		\smallskip
		
		\noindent\textbf{Case(C1).} For states whose reduced graph has $\scs-\xings-2$ edges, we consider the states $s$ where the edges of $\G_s'$ are not mixed first. For such a state
		\[
		Y(s)=(-q)^{-2s-\b(s)}[3]^{\scs-2}[2]^{\b(s)}=(-1)^{\b(s)}q^{2(\scs-\xings-2)}+\textup{lower degree terms}\,.
		\]
		Over a fixed pair of reduced edges with multiplicities $k_1$ and $k_2$ we sum these signs
		\[
		\sum_{j_1=1}^{k_{1}}\sum_{j_2=1}^{k_2} {k_1\choose j_1}{k_2\choose j_2}(-1)^{j_1+j_2}
		=
		\left(\sum_{j_1=1}^{k_{1}}{k_1\choose j_1}
		(-1)^{j_1}\right)\left(\sum_{j_2=1}^{k_2} {k_2\choose j_2}(-1)^{j_2}\right)=(-1)^2=1
		\]
		and the sum over all such pairs of reduced edges gives ${\edges'\choose 2}-\th$ where $\th$ is the number of pairs of mixed reduced edges.\\
		
		\noindent\textbf{Case(C2).} We now consider states where the edges of $\G_s'$ are over a mixed vertex of index $m$ and state mixing index $n\leq m$. First note that the web $\W(u_1,v_1,\dots, u_m,v_m)=\W(u_1',v_1',\dots, u_n',v_n')$ satisfies
		\[
			\W(u_1',v_1',\dots, u_n',v_n')=[2]^{\sum_{i=1}^n(u_i+v_i)-2n}\W(\underbrace{1,1,\dots, 1,1}_{\mbox{$2n$-terms}})\,.
		\]
		Observe that $\W(1,1,\dots,1,1)$ is exactly the closed sequence of squares given in Lemma \ref{lem.TraceSquare} for $k=2(n-1)$. Thus, together with the $\scs-3$ disjoint vertices
		\[
		\lla \W(u_1',v_1',\dots, u_n',v_n') \rra =[2]^{\sum_{i=1}^n(u_i+v_i)-2n}[3]\left([2]^2+\sum_{i=1}^{n-1}[2]^{2i}\right)\cdot [3]^{\scs-3}\,.
		\]
		In such a state $s$, there are $\b(s)=\sum_{i=1}^n(u_i+v_i)$ total $W$-resolutions. Therefore,
		\[
			Y(s)=(-1)^{\b(s)}q^{-2\xings-\b(s)}
			[2]^{\b(s)-2n}[3]^{\scs-2}\left([2]^2+\sum_{i=1}^{n-1}[2]^{2i}\right)
			\,.
		\]
		Thus
		\[
			d(s)=2(\scs-\xings-2-n)+\max(2, 2(n-1))=
			\begin{cases}
				2(\scs-\xings -2)& n=1
				\\
				2(\scs-\xings -3)& n>1
			\end{cases}
		\]
		and so only states which are semi-mixed contribute to the third coefficient. Now by Lemma \ref{lem.semisum}, the sum of the degree $2(\scs-\xings -2)$ terms in those states equals zero.
	In summary we have shown
	$$ \g_3={\scs+1\choose 2}+{\edges'\choose 2}-\scs \edges'+\mu -\th,$$
	which can be rewritten in the form given in the statement of the theorem.	
\end{proof}
	

	\section{Connected sum and disjoint union} \label{sec:split-connect}
	
	In this section we discuss the behavior of the coefficients $\g_1, \g_2, \g_3$ under disjoint union and connected sum of positive links.
	To avoid ambiguities we state the result on connected sums for knots only.  
		
We begin by recalling that by a result of Ozawa \cite[Theorem 1.4]{Ozawa} the decision of whether a positive link is non-split or prime can be made from its positive diagrams.
For terminology about link diagrams the reader is referred, for example, to \cite{Gutsbook}.
Let $L$ be a positive link and $D$ any positive diagram of $L$ that is reduced (i.e.\!  $D$ doesn't contain nugatory crossings).
Then  \cite[Theorem 1.4]{Ozawa}  gives the following.

\begin{itemize}
\item The link $L$ is prime if and only if $D$ is prime.
\item The link $L$ is non-split if and only if $D$ is non-split.
\end{itemize}
	
If a knot $K$ is the connected sum of prime knots $\connsum_{j=1}^p K_j$, then $p$ is the \emph{prime decomposition number} of $K$. The knots $K_i$ are the prime factors of $K$. Similarly, if a link $L$ is the disjoint union of non-split links $L_1,\cdots, L_s$ then  $s$ is the \emph{splitting number} of $L$.

If $K$ is a positive knot with prime factors  $K_1,\dots, K_p$, by  \cite[Theorem 1.4]{Ozawa}, $K$ admits a positive reduced diagram $D$ that is a connected sum $D=\connsum_{j=1}^p D_j$ where $D_j$ is a reduced positive diagram of $K_i$.

By Theorem 
\ref{thm.third}, we can compute  the quantities $\g_2(K)\coloneqq \scs(D)-\edges'(D)$ and $\lambda(K)\coloneqq \mu(D)-\theta(D)$ from $D$. Similarly,    $\gamma_2(K_j)=\scs_j-\edges'_j\coloneqq 
\scs(D_j)-\edges'(D_j)$ and 
$\lambda(K_j)\coloneqq \mu_j-\theta_j$, where $\mu_j-\theta_j\coloneqq \mu(D_j)-\theta(D_j)$ for $j=1\ldots, p$. Theorem \ref{thm.third} also implies that all these quantities are invariants of $K$ and $K_i$ respectively. Our next result shows that these invariants behave well under connected sums.
	
\begin{thm}\label{thm.connsum} Let $K$ be a positive knot with prime factors  $K_1,\dots, K_p$. The following hold:

\begin{enumerate}[(1)]
	
\item 	$\g_2(K)=1-p+\sum_{j=1}^p\g_2(K_j)$,
\smallskip

\item $\lambda(K)=\sum_{j=1}^p \lambda(K_j)$,

\smallskip

\item $\g_3(K)={\frac {(\g_2(K)+1)\,  \g_2(K)}{2}} + \lambda(K)$.
\end{enumerate}
\end{thm}
		
	\begin{proof}
		We proceed by induction on $p$. The base case $p=1$, is exactly Theorem \ref{thm.third}. 
		
		Inductively, assume the claim holds for all positive knots with prime decomposition number at most $p$.  Next, consider a positive knot with prime decomposition 
$\connsum_{j=1}^{p+1} K_j$, where by  \cite[Theorem 1.4]{Ozawa} each $K_i$ is a positive prime knot.

 Let $D$ be a positive diagram of $K=\connsum_{j=1}^{p} K_j$ and $D_{p+1}$ be  a prime, positive diagram of $K_{p+1}$.

 Consider the split link $L$ consisting of the disjoint union of  $K$ and $K_{p+1}$ with diagram $D'$ given by the disjoint union of diagrams $D$ and $D_{p+1}$.
  By computing the web invariant for each of $K$ and $K_{p+1}$ in $L$ disjoint from the other, we have $\lla L\rra=\lla K\rra\lla K_{p+1}\rra$.
   Write \[\lla K_{p+1}\rra= q^{2(\scs_{p+1}-\xings_{p+1})}+\g_2(K_{p+1})q^{2(\scs_{p+1}-\xings_{p+1})} +\g_{3}(K_{p+1}) q^{2(\scs_{p+1}-\xings_{p+1})} + \mbox{lower order terms}\]
		and by induction 
		\begin{align*}
			\lla K\rra =&~ q^{2(\scs-\xings)}+\g_2(K) q^{2(\scs-\xings-1)}+
			\left({\g_2(K)+1\choose 2}+\lambda(K)\right)q^{2(\scs-\xings-2)}
			\\&+ \mbox{lower order terms}\,.
		\end{align*}
		Expanding the product of invariants yields
		\[
		\lla K\rra\lla K_{p+1}\rra = q^{2(\scs''-\xings'')}+\g_2'q^{2(\scs''-\xings''-1)}+\g_3'q^{2(\scs''-\xings''-1)}+ \mbox{lower order terms}
		\]
		where $\scs''\coloneqq \scs+\scs_{p+1}$, $\xings'':=\xings+\xings_{p+1}$, $\g_2'\coloneqq \g_2(K)+\g_2(K_{p+1})$, $\lambda'\coloneqq \lambda+\lambda(K_{p+1})=\sum_{j=1}^{p+1}\lambda(K_j)$, and 
		\[
		\g_3'={\g_2(K)+1\choose 2}+{\g_2(K_{p+1})+1\choose 2}+\g_2(K)\, \g_2(K_{p+1})+\lambda'={\g_2(K)+\g_2(K_{p+1})+1\choose 2}+\lambda'\,.
		\]
		Now the HOMFLY skein relation and the specialization to the $\sl_3$ invariant in Equations \eqref{eq.skein} and \eqref{eq.homfly} imply that
		\begin{equation}\label{eq.connsum}
			[3]\lla K\#K_{p+1}\rra=\lla K\rra\lla K_{p+1}\rra\,.
		\end{equation}
		Expanding $[3]=q^2+1+q^{-2}$ now yields
		\[
		\lla K\#K_{p+1}\rra = q^{2(\scs''-\xings''-1)}
		+
		(\g_2'-1)q^{2(\scs''-\xings''-2)}+(\g_3'-\g_2')q^{2(\scs''-\xings''-3)}+ \mbox{lower order terms}\,.
		\]
		The second coefficient is given by
		\[
		\g_2'-1=\g_2(K)+\g_2(K_{p+1})-1=1-p+\sum_{j=1}^p \g_{2}(K_j)+\g_{2}(K_{p+1})-1=1-(p+1)+\sum_{j=1}^{p+1}  \g_{2}(K_j).
		\]
	Similarly, the third coefficient of $\lla K\#K_{p+1}\rra $ is equal to
		\begin{align*}
		{\g_2(K)+\g_2(K_{p+1})+1\choose 2}+\lambda'-\g_2(K)-\g_2(K_{p+1})={\g_2(K)+\g_2(K_{p+1})\choose 2}+\lambda'=
		{\g_2'\choose 2}+\lambda'
		\end{align*}
		which completes the induction.
	\end{proof}

\begin{rem} The proof of Theorem \ref{thm.connsum}  easily adapts to show that if a positive link $L$ is a connected sum of positive prime links $L_1\ldots L_k$ then formulae
$(1)-(3)$ hold.
\end{rem}

	\medskip

	
	\section{Obstructing positive braiding} \label{sec:braids}
	A particularly interesting class of positive links is the class of links that can be represented as closures of positive braids.
	It been known for a long time that positive closed braids are fibered \cite{Stallings} and that not all positive knots are fibered nor are all positive fibered knots the closure of positive braids. 
	
	Positive braids have been studied extensively in low dimensional topology and the question  of determining which links can be represented by positive closed braids has been studied considerably.
	For instance, recently, they have been surfaced in the studies around the ``$L$-space conjecture" that relates taut foliations on 3-manifolds to Floer theoretic invariants (see \cite{Kr} and references therein).
	
	\begin{thm}\label{braid} The following hold:
	\begin{enumerate} [(1)]
	
	\item if a link $L$ of splitting number $s$ is the closure of a positive closed braid, then $\g_1(L)=1$ and  $\g_2(L)=s$,
	
	\smallskip
	
	\item if a knot  $K$ of prime decomposition number $p$ is a positive closed braid, then $\g_1(K)=\g_2(K)=1$  and $\g_3(K)=p+1$.
	\end{enumerate}
	\end{thm}
	\begin{proof} We begin with the proof of claim (1). \ By \cite{Ozawa} (see also \cite{positiveprime}) $L$ has a diagram $D$ as a closed braid that realizes the splitting number $s$.
	That is, $D$ is the disjoint union of closed {positive} braid diagrams $D_1, \ldots, D_s$, where $D_j$ represents a link $L_j$.
	By Theorem \ref{thm.third} we have $\g_1(L)=1$ and $\g_2(L)=\scs(D)-\edges'(D)=\chi(\G'_W(D))$.
	Now we have  $$\g_2(L)=\chi(\G'_W(D))=\sum_{j=1}^s \chi(\G'_W(D_j))=\sum_{j=1}^s \g_2(L_i)= s,$$
	where the last equality follows from the fact that $L_i$ is fibered \cite{Stallings} and by Theorem \ref{thm.fiber}.
	\smallskip
	
	Now we prove part (2). The claim that  $\g_1(K)=\g_2(K)=1$ follows from part (1).
	Again by \cite{Ozawa, positiveprime}, $K$ has a diagram $D$  that can be written as a connected sum of $p$ prime positive closed braids.

	Suppose $D$ is prime, i.e.\! $p=1$.
	By direct examination we can see that the graph ${\mathbb G}'_S(D)$ is a tree.
	Primeness implies that there is an edge in the Seifert graph between each adjacent Seifert circle determined by the braid. It also implies that each pair of adjacent reduced edges is mixed. Thus 
	$\mu=\edges'=\scs-1$ and $\th=\scs-2$. We now have $\g_3=2$, by Theorem \ref{thm.third}. Note also that $\lambda(K)=\mu-\th=1$.
	
	For $p>1$, $\lambda(K_i)=1$ for each prime factor of $K$. Now apply Theorem \ref{thm.connsum}(2) to get  $\lambda(K)=p$. Then by Theorem \ref{thm.connsum}(3) and since $\g_2(K)=1$, we get $\g_3(K)=p+1$.
	\end{proof}
	
		The following generalizes  \cite[Corollary 3]{alternatingtorus}.
		
		\begin{cor} \label{altbraid}Let $L$ be a non-split link that admits an alternating positive closed braid diagram.
		Then $L$ is a
		connected 
		sum of  $(2,n)$ torus links.
		\end{cor}
		\begin{proof}
By our earlier discussion, a reduced  alternating positive closed braid diagram of $L$ will be either prime or a connected sum of
prime alternating  positive closed braid diagrams. Hence it is sufficient to prove that the only prime links that admit  alternating  positive closed braid diagrams are $(2,n)$ torus links.

  Let $D=D(L)$ be a prime alternating positive closed braid  diagram. By Theorem \ref{braid}, $\g_2=1$ and $\g_3=2$.
	Since $D$ is alternating, $\theta=0$. Then $\g_3=1+\mu-\theta=2$ implies that $\mu=1$ and hence	exactly one edge in the reduced graph has multiplicity more than one in the unreduced Seifert graph. Since the reduced Seifert graph is a tree and $D$ is prime, $D$ consists of a pair of circles joined in a single twist region. Thus it is the standard 2-braid closure of a  $(2,n)$ torus link for some $n>1$.
\end{proof}

	Theorem \ref{braid} should be compared with the work of Ito \cite{Ito}, who finds obstruction to positive braiding in terms of the coefficients of the HOMFLY link polynomial.
\smallskip

\subsection{Concluding Remarks}
We have verified Theorem \ref{thm.third} for knots up to twelve crossings by evaluating the specializations of the HOMFLY polynomial from KnotInfo \cite{knotinfo}. 
The $\Atwo$ invariant detects braid positivity among the 33 positive fibered prime knots with at most twelve crossings. With the exception of $\mathsf{11_{n183}}$, $\g_3=1$ on all knots which have a positive diagram but are not the closure of a positive braid. These knots are tabulated in the table below.
	
\medskip

	\begin{table}[h!]
		\centering
		\label{table.g3}
		\vspace{-1\baselineskip}
		
		\[
		\begin{array}{c|c|c}
			\mbox{Knot} & \mbox{Positive Braid} & \g_3 \\\hline
			\mathsf{3_1} & \mbox{Y} & 2\\\hline
			\mathsf{5_1} & \mbox{Y} & 2\\\hline
			\mathsf{7_1} & \mbox{Y} & 2\\\hline
			\mathsf{8_{19}} & \mbox{Y}& 2\\\hline
			\mathsf{9_1} & \mbox{Y} & 2\\\hline
			\mathsf{10_{124}} & \mbox{Y} & 2\\\hline
			\mathsf{10_{139}} & \mbox{Y} & 2\\\hline
			\mathsf{10_{152}} & \mbox{Y} & 2\\\hline
			\mathsf{10_{154}} & \mbox{N} & 1\\\hline
			\mathsf{10_{161}} & \mbox{N} & 1\\\hline
			\mathsf{11_{a367}} & \mbox{Y} & 2
		\end{array}
		\qquad
		\begin{array}{c|c|c}
			\mbox{Knot} & \mbox{Positive Braid} & \g_3 \\\hline
			\mathsf{11_{n77}} & \mbox{Y} & 2\\\hline		
			\mathsf{11_{n183}} & \mbox{N} & 0\\\hline
			\mathsf{12_{n91}} & \mbox{N} & 1\\\hline
			\mathsf{12_{n105}} & \mbox{N} & 1\\\hline
			\mathsf{12_{n136}} & \mbox{N} & 1\\\hline
			\mathsf{12_{n187}} & \mbox{N} & 1\\\hline
			\mathsf{12_{n242}} & \mbox{Y} & 2\\\hline
			\mathsf{12_{n328}} & \mbox{N} & 1\\\hline
			\mathsf{12_{n417}} & \mbox{N} & 1\\\hline
			\mathsf{12_{n426}} & \mbox{N} & 1\\\hline
			\mathsf{12_{n472}} & \mbox{Y} & 2
		\end{array}
		\qquad
		\begin{array}{c|c|c}
			\mbox{Knot} & \mbox{Positive Braid} & \g_3 \\\hline
			\mathsf{12_{n518}} & \mbox{N} & 1\\\hline
			\mathsf{12_{n574}} & \mbox{Y} & 2\\\hline
			\mathsf{12_{n591}} & \mbox{N} & 1\\\hline
			\mathsf{12_{n640}} & \mbox{N} & 1\\\hline
			\mathsf{12_{n647}} & \mbox{N} & 1\\\hline
			\mathsf{12_{n679}} & \mbox{Y} & 2\\\hline
			\mathsf{12_{n688}} & \mbox{Y} & 2\\\hline
			\mathsf{12_{n694}} & \mbox{N} & 1\\\hline
			\mathsf{12_{n725}} & \mbox{Y} & 2\\\hline
			\mathsf{12_{n850}} & \mbox{N} & 1\\\hline
			\mathsf{12_{n888}} & \mbox{Y} & 2
		\end{array}
		\]
	\end{table}
	
\begin{example}As an example, we compute $\mathsf{11_{n183}}$ from the positive diagram in Figure \ref{fig.11n183}, generated using SnapPy \cite{SnapPy} from the presentation in \cite{unknotting}. We can see that $\scs=7$, $\edges'=6$, $\mu=6$, and $\th=7$. Thus, indeed $\g_2=1$ and $\g_3=0$.
\end{example}
	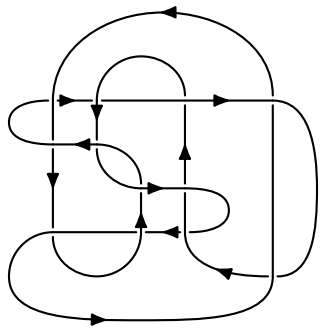
\begin{figure}[h!]
		\begin{tikzpicture}[very thick, line cap=round, line join=round,scale=0.42]
			\begin{scope}
				\draw[] (1.89, 5.27) .. controls (1.23, 5.27) and (0.50, 5.40) .. 
				(0.50, 5.97) .. controls (0.50, 6.49) and (1.14, 6.66) .. (1.74, 6.66);
				\draw[mid>] (2.04, 6.66) .. controls (2.41, 6.66) and (2.77, 6.66) .. (3.13, 6.66);
				\draw[] (3.44, 6.66) .. controls (4.31, 6.66) and (5.19, 6.66) .. (6.07, 6.66);
				\draw[mid>] (6.07, 6.66) .. controls (7.00, 6.66) and (7.93, 6.66) .. (8.85, 6.66);
				\draw[] (8.85, 6.66) .. controls (10.00, 6.66) and (10.25, 5.21) .. 
				(10.25, 3.88) .. controls (10.25, 2.57) and (10.13, 1.09) .. (9.00, 1.09);
				\draw[mid>] (8.70, 1.09) .. controls (7.43, 1.09) and (6.07, 1.39) .. (6.07, 2.49);
				\draw[] (6.07, 2.49) .. controls (6.07, 2.90) and (6.07, 3.31) .. (6.07, 3.73);
				\draw[mid>] (6.07, 4.03) .. controls (6.07, 4.86) and (6.07, 5.69) .. (6.07, 6.51);
				\draw[] (6.07, 6.82) .. controls (6.07, 7.53) and (5.42, 8.06) .. 
				(4.68, 8.06) .. controls (3.91, 8.06) and (3.28, 7.43) .. (3.28, 6.66);
				\draw[mid>] (3.28, 6.66) .. controls (3.28, 6.25) and (3.28, 5.84) .. (3.28, 5.42);
				\draw[] (3.28, 5.12) .. controls (3.28, 4.41) and (3.93, 3.88) .. (4.68, 3.88);
				\draw[mid>] (4.68, 3.88) .. controls (5.14, 3.88) and (5.60, 3.88) .. (6.07, 3.88);
				\draw[] (6.07, 3.88) .. controls (6.74, 3.88) and (7.46, 3.76) .. 
				(7.46, 3.18) .. controls (7.46, 2.66) and (6.82, 2.49) .. (6.22, 2.49);
				\draw[mid>] (5.92, 2.49) .. controls (5.55, 2.49) and (5.19, 2.49) .. (4.83, 2.49);
				\draw[] (4.53, 2.49) .. controls (3.65, 2.49) and (2.77, 2.49) .. (1.89, 2.49);
				\draw[mid>] (1.89, 2.49) .. controls (1.12, 2.49) and (0.50, 1.86) .. 
				(0.50, 1.09) .. controls (0.50, -0.30) and (2.77, -0.30) .. 
				(4.68, -0.30) .. controls (6.58, -0.30) and (8.85, -0.30) .. (8.85, 1.09);
				\draw[] (8.85, 1.09) .. controls (8.85, 2.90) and (8.85, 4.71) .. (8.85, 6.51);
				\draw[mid>] (8.85, 6.82) .. controls (8.85, 8.45) and (7.16, 9.45) .. 
				(5.37, 9.45) .. controls (3.56, 9.45) and (1.89, 8.35) .. (1.89, 6.66);
				\draw[] (1.89, 6.66) .. controls (1.89, 6.25) and (1.89, 5.84) .. (1.89, 5.42);
				\draw[mid>] (1.89, 5.12) .. controls (1.89, 4.29) and (1.89, 3.47) .. (1.89, 2.64);
				\draw[] (1.89, 2.34) .. controls (1.89, 1.62) and (2.54, 1.09) .. 
				(3.28, 1.09) .. controls (4.05, 1.09) and (4.68, 1.72) .. (4.68, 2.49);
				\draw[mid>] (4.68, 2.49) .. controls (4.68, 2.90) and (4.68, 3.31) .. (4.68, 3.73);
				\draw[] (4.68, 4.03) .. controls (4.68, 4.75) and (4.03, 5.27) .. (3.28, 5.27);
				\draw[mid>] (3.28, 5.27) .. controls (2.82, 5.27) and (2.36, 5.27) .. (1.89, 5.27);
			\end{scope}

		\end{tikzpicture}
		\caption{A positive 12-crossing  diagram of the knot $\mathsf{11_{n183}}$.}\label{fig.11n183}
	\end{figure}
	
\begin{rem} We note that the invariant $\g_3$ does not detect all positive closed braids. For example, the authors in \cite{unknotting} define a family $\{{K}_n\}_{n\geq 1}$ of positive fibered knots. They use  \cite[Theorem 1.1]{Ito}, which relies on positivity properties of the 2-variable HOMFLY polynomial, to show that none of these knots is a closed positive braid. The knot $K_n$ has a positive diagram $D_n$ defined as cyclic Conway sum of  $2n+1$ copies of the tangle  shown  in \cite[Figure 8]{unknotting}.
Using these diagrams one can directly compute $\scs(K_n)=6n+5$, $\edges'(K_n)=6n+4$,  $\mu(D_n)=6n+4$, and $\th(D_n)=6n+3$. This gives $\g_3(K_n)=2$, for all $n\geq 1$.
\end{rem}	
\medskip

\bibliographystyle{alpha}
	\bibliography{references}		
	\end{document}

%% file: preamble.tex
\usepackage{amsmath,amssymb,amsfonts,amsthm,mathrsfs,mathtools,bbm,cancel,calc,arydshln, tikz-cd,tikz,color,enumerate,scalerel}
\usepackage[all]{xy}

\usetikzlibrary{positioning}
\usetikzlibrary{calc}
\usetikzlibrary{decorations.markings}
\usetikzlibrary{arrows}
\usetikzlibrary{calc}
\usetikzlibrary{decorations.markings}

\tikzset{> /.tip = {Latex[round,length=.25cm]}
}

\tikzset{every path/.style={thick}
}
\tikzset{baseline={([yshift=-.5ex]current bounding box.center)} }

\tikzstyle{mid>}=[decoration={markings, mark=at position 0.5 with {\arrow{>}}}, postaction={decorate}]
\tikzstyle{mid<}=[decoration={markings, mark=at position 0.5 with {\arrowreversed{>}}}, postaction={decorate}]
\tikzstyle over=[preaction={draw,line width=5pt,white}]

\usepackage[bookmarks=true,%
colorlinks=true,%
linkcolor=blue,%
citecolor=blue,%
filecolor=blue,%
menucolor=blue,%
urlcolor=blue,%
breaklinks=true]{hyperref}

\def\la{\langle}
\def\ra{\rangle}
\def\lla{\la\!\la}
\def\rra{\ra\!\ra}

\def\Atwo{\mathfrak{sl}_3}
\def\bZ{\mathbb{Z}}
\def\bZq{\mathbb{Z}[q,q^{-1}]}

\def\bG{\mathbb{G}}
\def\bW{\mathbb{W}}
\def\G{\bG}
\def\W{\bW}
\def\edges{\mathbb{E}}

\def\lifts{\mathbb{T}}

\def\xings{e}
\def\scs{v}
\def\edges{e}
\def\a{\alpha}
\def\b{\beta}
\def\th{\theta}
\def\g{\gamma}

\def\sl{\mathfrak {sl}}
\DeclareMathOperator*{\connsum}{\scalerel*{\#}{\sum}}

\definecolor{ForestGreen}{RGB}{34,139,34}

\theoremstyle{plain}
\newtheorem{prop}{Proposition}[section]
\makeatletter
\@addtoreset{prop}{section}
\makeatother

\newtheorem{thm}[prop]{Theorem}
\newtheorem{lem}[prop]{Lemma}
\newtheorem*{lem*}{Lemma}
\newtheorem{cor}[prop]{Corollary}
\newtheorem*{cor*}{Corollary}           
\newtheorem*{conv*}{Convention}

\theoremstyle{remark}
\newtheorem{remark}[prop]{Remark}

\newtheorem{defn}[prop]{Definition}

\newtheorem{example}[prop]{Example}

\newenvironment{rem}
{\pushQED{\qed}\remark}
{\popQED\endremark}

\theoremstyle{plain}

%% file: main.bbl
\newcommand{\etalchar}[1]{$^{#1}$}
\begin{thebibliography}{CDGW}

\bibitem[Baa13]{alternatingtorus}
S.~Baader.
\newblock Positive braids of maximal signature.
\newblock {\em Enseign. Math.}, 59(3-4):351--358, 2013.

\bibitem[CDGW]{SnapPy}
M.~Culler, N.~M. Dunfield, M.~Goerner, and J.~R. Weeks.
\newblock Snap{P}y, a computer program for studying the geometry and topology
  of $3$-manifolds.
\newblock Available at \url{http://snappy.computop.org}.

\bibitem[Cro89]{homogeneous}
P.~R. Cromwell.
\newblock Homogeneous links.
\newblock {\em J. London Math. Soc. (2)}, 39(3):535--552, 1989.

\bibitem[Cro93]{positiveprime}
P.~R. Cromwell.
\newblock Positive braids are visually prime.
\newblock {\em Proc. London Math. Soc. (3)}, 67(2):384--424, 1993.

\bibitem[DL06]{DL}
O.~T. Dasbach and X.-S. Lin.
\newblock On the head and the tail of the colored {J}ones polynomial.
\newblock {\em Compos. Math.}, 142(5):1332--1342, 2006.

\bibitem[FKP13]{Gutsbook}
David Futer, Efstratia Kalfagianni, and Jessica Purcell.
\newblock {\em Guts of surfaces and the colored {J}ones polynomial}, volume
  2069 of {\em Lecture Notes in Mathematics}.
\newblock Springer, Heidelberg, 2013.

\bibitem[Fut13]{Futer}
David Futer.
\newblock Fiber detection for state surfaces.
\newblock {\em Algebr. Geom. Topol.}, 13(5):2799--2807, 2013.

\bibitem[Ito22]{Ito}
T.~Ito.
\newblock A note on {HOMFLY} polynomial of positive braid links.
\newblock {\em Internat. J. Math.}, 33(4):Paper No. 2250031, 18, 2022.

\bibitem[Kaw96]{AKbook}
A.~Kawauchi.
\newblock {\em A survey of knot theory}.
\newblock Birkh\"auser Verlag, Basel, 1996.
\newblock Translated and revised from the 1990 Japanese original by the author.

\bibitem[KLM{\etalchar{+}}]{unknotting}
M.~Kegel, L.~Luwark, N.~Manikandan, F.~Misev, L.~Mousseau, and M.~Silvero.
\newblock On unknotting fibered positive knots and braids.
\newblock Preprint 2023, \href{https://arxiv.org/abs/2312.07339}{
  arXiv:2312.07339}.

\bibitem[Kri25]{Kr}
S.~Krishna.
\newblock Taut foliations, braid positivity, and unknot detection.
\newblock {\em Adv. Math.}, 470:Paper No. 110233, 111, 2025.

\bibitem[Kup96]{Kuperberg}
G.~Kuperberg.
\newblock Spiders for rank {$2$} {L}ie algebras.
\newblock {\em Comm. Math. Phys.}, 180(1):109--151, 1996.

\bibitem[LM25]{knotinfo}
C.~Livingston and A.~H. Moore.
\newblock {KnotInfo: Table of Knot Invariants}.
\newblock \url{knotinfo.math.indiana.edu}, July 2025.

\bibitem[Oht02]{Ohtsuki}
T.~Ohtsuki.
\newblock {\em Quantum invariants}, volume~29 of {\em Series on Knots and
  Everything}.
\newblock World Scientific Publishing Co., Inc., River Edge, NJ, 2002.
\newblock A study of knots, 3-manifolds, and their sets.

\bibitem[Oza02]{Ozawa}
M.~Ozawa.
\newblock Closed incompressible surfaces in the complements of positive knots.
\newblock {\em Comment. Math. Helv.}, 77(2):235--243, 2002.

\bibitem[RT90]{RT}
N.~Yu. Reshetikhin and V.~G. Turaev.
\newblock Ribbon graphs and their invariants derived from quantum groups.
\newblock {\em Comm. Math. Phys.}, 127(1):1--26, 1990.

\bibitem[Sta78]{Stallings}
J.~R. Stallings.
\newblock Constructions of fibred knots and links.
\newblock In {\em Algebraic and geometric topology ({P}roc. {S}ympos. {P}ure
  {M}ath., {S}tanford {U}niv., {S}tanford, {C}alif., 1976), {P}art 2}, volume
  XXXII of {\em Proc. Sympos. Pure Math.}, pages 55--60. Amer. Math. Soc.,
  Providence, RI, 1978.

\end{thebibliography}
